\documentclass[final]{siamltex}

\usepackage{amsmath}
\usepackage{amsfonts}
\usepackage{amssymb}
\usepackage{amsmath}

\newtheorem{Lemma}     {lemma}[section]
\newtheorem{Theorem}   {theorem}
\newtheorem{Proposition}      [Lemma]{Proposition}
\newtheorem{Corollary}    [Lemma]{Corollary}
\newtheorem{Remark}    [Lemma]{Remark}

\usepackage{color}

\definecolor{orange}{rgb}{0,0,0}

\newcommand{\alert}[1]{{\color{orange}{#1}}}
\newcommand{\TT}{\mathrm T}

\newcommand{\nada}[1]   {}

\newcommand{\R}         {\ensuremath{\mathbb R}}

\newcommand{\cd}        {{c^\dagger}}

\newcommand{\eps}       {\varepsilon}
\newcommand{\Om}        {\Omega}

\begin{document}

\title{Global exponential convergence to variational traveling waves
  in cylinders}

\author{C. B. Muratov\thanks{Department of Mathematical Sciences, New
    Jersey Institute of Technology, Newark, NJ 07102, USA ({\tt
      muratov@njit.edu}). The work of this author was supported by NSF
    via grants DMS-0718027 and DMS-0908279.} \and
  M. Novaga\thanks{Dipartimento di Matematica, Universit\`a di Padova,
    Via Trieste 63, 35121 Padova, Italy ({\tt novaga@dm.unipd.it}).}}

\maketitle

\begin{abstract}
  We prove, under generic assumptions, that the special variational
  traveling wave that minimizes \alert{the exponentially weighted
    Ginzburg-Landau} functional \alert{associated with scalar
    reaction-diffusion equations in infinite cylinders} is the
  long-time attractor for the solutions of the initial value problems
  with front-like initial data. The convergence to this traveling wave
  is exponentially fast. The obtained result is mainly a consequence
  of the gradient flow structure of the considered equation in the
  exponentially weighted spaces and does not depend on the precise
  details of the problem. It strengthens our earlier generic
  propagation and selection result for ``pushed'' fronts.
\end{abstract}



\pagestyle{myheadings} \thispagestyle{plain} \markboth{C. B. MURATOV
  AND M. NOVAGA}{CONVERGENCE TO VARIATIONAL TRAVELING WAVES}

\numberwithin{equation}{section}

\section{Introduction}

One of the most fundamental problems in the theory of
reaction-diffusion equations has to do with the long-time asymptotic
behavior of solutions of the associated initial value problem on
unbounded domains \cite{fife79,volpert,xin00}. In its simplest form,
it may be formulated for a one-dimensional scalar reaction-diffusion
equation
\begin{eqnarray}
  \label{eq:1d}
  u_t = u_{xx} + f(u), \qquad u: \mathbb R \times \mathbb R^+ \to [0,
  1], 
\end{eqnarray}
with an unbalanced bistable nonlinearity $f(u)$, i.e., when $f$ is a
smooth function which has precisely three non-degenerate zeros in $[0,
1]$, with
\begin{eqnarray}
  \label{eq:8}
  f(0) = f(1) = 0, \qquad f'(0) < 0, \quad f'(1) < 0, \qquad \int_0^1
  f(u) du > 0,  
\end{eqnarray}
e.g. $f(u) = u (1 - u) (u - \tfrac14)$. For such an equation, it was
first proved by Kanel' that initial data $u(x, t) = u_0(x)$ with the
property that $u_0(x) = 0$ for all $x > b$, $u_0(x) = 1$ for all $x <
a$, and $u_0(x)$ is monotone decreasing for $x \in (a, b)$, with some
$-\infty < a < b < +\infty$, converges uniformly to a (unique up to
translations) traveling wave solution, i.e., a solution $u(x, t) =
\bar u(x - ct)$ of (\ref{eq:1d}), with some uniquely determined speed
$c > 0$, connecting monotonically $u = 0$ at $x = +\infty$ with $u =
1$ at $x = -\infty$, in a reference frame moving with speed $c$
\cite{kanel60,kanel62}. In a subsequent work, Fife and McLeod extended
this result to a much wider class of initial data and also showed that
the convergence is exponentially fast \cite{fife77}. Qualitatively,
the conclusion of these analyses is that the solution of the
considered initial value problem with front-like initial data
converges exponentially fast to a traveling front invading the ``less
stable'' equilibrium $u = 0$ by a ``more stable'' equilibrium $u =
1$. We note that a similar result was proved for a certain class of
monostable nonlinearities \cite{rothe81}, but it does {\em not} hold
(in the reference frame moving with constant speed and in the sense of
exponential convergence) in the case of the Fisher's equation
\cite{kolmogorov37,uchiyama77,bramson,kirchgaessner92}.

In the multi-dimensional setting, these kinds of results were
\alert{subsequently} obtained for initial boundary value problems for
equations in infinite cylindrical domains:
\begin{equation}\label{pde}
u_t = \Delta u + f(u, y), \qquad u(x, 0) = u_0(x),
\end{equation}
where $u:\Sigma \times \mathbb R^+ \to \mathbb R$, $\Sigma = \Omega
\times \mathbb R \subset \mathbb R^n$, $\Omega \subset \mathbb
R^{n-1}$ is a bounded domain with sufficiently smooth boundary,
$f:\mathbb R \times \Omega \to \mathbb R$ is a nonlinear reaction
term, with either Neumann or Dirichlet boundary conditions. By $x =
(y, z) \in \Sigma$, we always denote a point with coordinate $y \in
\Omega$ on the cylinder cross-section and $z \in \mathbb R$ along the
cylinder axis. More generally, one can consider either Dirichlet or
Neumann boundary conditions on different connected portions of
$\partial\Omega$:
\begin{equation}
  \label{bc}
  u \bigl|_{\partial \Sigma_\pm} = 0, \qquad \nu \cdot
  \nabla u \bigl|_{\partial \Sigma_0} = 0,
\end{equation}
where $\partial \Sigma_\pm = \partial \Omega_\pm \times \mathbb R$ and
$\partial \Sigma_0 = \partial \Omega_0 \times \mathbb R$, allowing for
more than one connected component for $\partial \Omega$ (for
motivation and further discussion of the boundary conditions, see
\cite{mn:cms08,mn:cvar08}). \alert{Note that transverse advection by a
  potential flow can also be straightforwardly included in the present
  treatment, as was done in \cite{mn:cms08,mn:cvar08}. For simplicity
  of presentation, in this paper we do not consider the advection term
  and concentrate on pure reaction-diffusion problems.}

Without loss of generality, we may assume
that $u = 0$ is a trivial solution of (\ref{pde}) and consider
traveling waves that invade the $u = 0$ equilibrium, i.e., the
solutions of (\ref{pde}) and (\ref{bc}) in the form $u(x, t) = \bar
u(y, z - ct)$, for some $c > 0$, which converge to zero uniformly as
$z \to +\infty$. These solutions satisfy the elliptic equation
\begin{eqnarray}\label{trav}
  \Delta \bar u + c \bar u_z + f(\bar u, y) = 0,
\end{eqnarray}
together with the respective boundary conditions in (\ref{bc}) (by a
solution, we mean a pair $(c, \bar u)$, with $\bar u \in C^2(\Sigma)
\cap C^1(\overline \Sigma)$ being a classical solution of \eqref{trav}
and \eqref{bc}). We refer to \cite{berestycki92,volpert,mn:cms08} and
references therein, for a comprehensive treatment of the subject of
traveling waves. In particular, under certain specific assumptions one
obtains uniqueness (up to translations) and global exponential
convergence to these solutions for the initial value problem with
front-like initial data \cite{mallordy95,roquejoffre94,roquejoffre97}
(see \alert{the end of} Sec. \ref{sec:vari-form-main} for a more
detailed discussion and a comparison with the present results). This
property, therefore, indicates the ubiquitous role of the traveling
fronts in the behavior of the solutions of (\ref{pde}).

Since in general (\ref{trav}) may have many solutions, an important
question is which of these solutions, if any, can be a long-time limit
of the evolution governed by (\ref{pde}), for a given class of initial
data. As was recently pointed out in \cite{m:dcdsb04}, in the case of
initial data with sufficiently fast exponential decay at $z = +\infty$
the relevant class of traveling wave solutions consists of the
so-called variational traveling waves, even for systems of
reaction-diffusion equations in which the nonlinearity is a
gradient. More recently, we showed that a special class of variational
traveling wave solutions that minimize \alert{the exponentially
  weighted Ginzburg-Landau} functional (see
Sec. \ref{sec:vari-form-main} for precise definitions and statements)
are relevant for the long-time behavior of the initial value problem
in the sense of propagation of the leading edge and, in particular,
determine the propagation speed for front-like initial data
\cite{mn:cms08}. It is then natural to ask whether these special
traveling fronts are also the long-time attractors for the solutions
of \eqref{pde} in the moving reference frame. In this paper, we give a
positive answer to this question under a few extra non-degeneracy
assumptions to those of \cite{mn:cms08} which hold generically in the
considered class of problems.

Our paper is organized as follows. In Sec. \ref{sec:vari-form-main},
we introduce the variational formulation for the traveling waves of
interest, state the main result and compare it with those available in
the literature. In Sec. \ref{secvar}, we list and discuss our
assumptions, as well as state a number of auxiliary results used in
the paper. In Sec. \ref{secsta}, we perform local stability analysis
of the traveling waves of interest in the exponentially weighted
Sobolev spaces, and in Sec. \ref{sec:main} we prove convergence to the
traveling wave in the large, completing the proof of the main theorem.

\smallskip

\paragraph{Some notation} For every $-\infty \leq a < b \leq +\infty$
and $c > 0$, the symbol $L^2_c(\Omega \times (a, b))$ denotes the
Hilbert space of all functions $u : \Omega \times (a, b) \to \mathbb
R$ with $||u||_{L^2_c(\Omega \times (a, b))}^2 = \int_a^b \int_\Omega
e^{cz} u^2(y, z) \, dy \, dz$. Likewise, by $L^2_c(\Sigma)$,
$H^1_c(\Sigma)$ and $H^2_c(\Sigma)$, we denote the spaces of functions
which are square integrable with the above exponential weight,
together with their first and second derivatives, respectively, in
$\Sigma$. We also use the symbol $C_b(A)$ to denote the space of
bounded continuous function on $A$ equipped with the sup-norm.  In all
statements and proofs the constants are always assumed to implicitly
depend on $f$, $\Omega$ and the choice of the boundary conditions.  In
the proofs the numbers $C, M$, etc., may change from line to line. We
will also use the symbol $\TT_R$ to denote a translation by $R$ along
the $z$-axis, i.e., $\TT_R u(\cdot, z) = u(\cdot, z - R)$.

\section{Variational formulation and main result}
\label{sec:vari-form-main}

~

\alert{The fact that \eqref{trav} possesses a variational structure in
  exponentially weighted Sobolev spaces was, to our knowledge, first
  pointed out by Heinze \cite{heinze89,heinze} (see also
  \cite{fife77,volpert,roquejoffre94,lmn:cpam04,gallay07,mn:cms08} in
  the context of \eqref{pde}, and
  \cite{m:dcdsb04,lmn:arma08,risler08,gallay09} in the context of its
  extensions).}  As we \alert{recently} showed in \cite{mn:cms08}, for
scalar reaction-diffusion equations considered here the solution of
\eqref{trav} which determines the asymptotic speed of propagation with
front-like initial data is a special variational traveling wave which
is the minimizer of the \alert{the exponentially weighted
    Ginzburg-Landau} functional
\begin{eqnarray}
  \label{eq:Phic}
  \Phi_c[u] := \int_\Sigma e^{cz} \left( \frac{1}{2} |\nabla u|^2 +
    V(u, y) \right) \, dx  
    \qquad c>0,
\end{eqnarray}
where 
\begin{eqnarray}
  \label{eq:V}
  V(u,y) = -\int_0^u f(s,y) \chi_{[0,1]}(s) \, ds, \qquad
  \chi_{[0,1]}(s) = 
  \begin{cases}
    1, & s \in [0, 1] \\
    0, & s \not\in [0, 1]
  \end{cases},
\end{eqnarray}
over all functions lying in the exponentially weighted Sobolev space
$H^1_c(\Sigma)$. We point out that such a minimizer can only exist for
a specific value of $c=c^\dag > 0$ (see Theorem \ref{t:min} below).
Under quite general assumptions on the potential $V$, in \cite[Theorem
5.8]{mn:cms08} we proved that the asymptotic speed of propagation of
solutions to \eqref{pde} is precisely given by $c^\dag$, assuming that
the initial datum is front-like, i.e., if it stays sufficiently far
away from zero as $z \to -\infty$ and decays sufficiently fast to zero
as $z\to +\infty$.  In this paper, we discuss the local and global
stability of such variational traveling waves.

\alert{Our main result is contained in the following theorem (for the
  details of the definitions and hypotheses, see Sec. \ref{secvar}):}

\begin{Theorem}\label{t:main}
  Assume hypotheses (H1)--(H3) and (N1)--(N2) are satisfied, and let
  $c^\dag$, $\bar u$, $v$ be as in Theorem \ref{t:min}. Then there
  exist $\alpha > 0$ and $\sigma > 0$, such that if $u_0 \in
  C^0(\overline \Sigma) \cap W^{1,\infty}(\Sigma) \cap
  L^2_\cd(\Sigma)$ satisfies $0\le u_0\le 1$ and
  \begin{equation}\label{inidat}
    \liminf_{z\to -\infty}u_0(\cdot,z)\ge v - \alpha \qquad\ {\rm
      uniformly\ in\ }\Omega\,, 
  \end{equation} 
  there exists $R_\infty \in \mathbb R$, such that if $u$ is the
  solution of (\ref{pde}) and (\ref{bc}) with initial datum $u_0$,
  then
  \begin{eqnarray}\label{eqmain}
    ||\TT_{ R_\infty-c^\dag t} u(\cdot, t) - \bar
    u||_{H^2_\cd(\Sigma)} \leq C e^{-\sigma t} 
  \end{eqnarray}
  for every $t \geq t_0$, with arbitrary $t_0 > 0$ and some $C > 0$
  independent of $t$.
\end{Theorem}

Note that by Proposition \ref{p:ini} below we know that $u(\cdot,t)$
is bounded in $W^{2,p}(\Omega\times [M,M+1])$ uniformly in $M \in
\mathbb R$ and $t\in [t_0,+\infty)$, for all $t_0>0$ and
$p<\infty$. Since this bound also applies to $\bar u$, from
\eqref{eqmain} we get the following

\begin{Corollary}
  \label{remlip}
  In the statement of Theorem \ref{t:main}, the inequality
  \eqref{eqmain} may be replaced with
  \begin{eqnarray}
    \label{eq:2}
    \|\TT_{R_\infty - \cd t} u(\cdot, t) - \bar
    u\|_{C^1(\overline\Omega\times [z_0, z_1])} \leq C e^{-\sigma t},  
  \end{eqnarray}
  for all $z_0 < z_1$, $t \geq t_0 > 0$, and some $C > 0$ independent
  of $t$ and $z_1$.
\end{Corollary}

\alert{Let us point out that} the upper bound $u_0\le 1$ in Theorem
\ref{t:main} can be replaced with the condition $u_0(\cdot, z) \le
\bar v$ for every $z \in \mathbb R$, where $\bar v \in C^2(\Omega)
\cap C^1(\overline\Omega)$ satisfies
\begin{equation}\label{superpsi}
  \bar v > 0, \quad \Delta_y \bar v + f(\bar v,y)\le 0 \qquad {\rm
    for\  all\ }y\in 
  \Omega, 
\end{equation}
together with the boundary conditions from \eqref{bc}.  In this case,
the condition $f(1,y)\le 0$ in assumption (H1) below \alert{should} be
replaced by \eqref{superpsi}, the conditions in (H2) \alert{should
  hold for $0 \leq u \leq \bar v$, and the} definition of $V$
\alert{in \eqref{eq:V}} should be modified accordingly. We note that,
in particular, one can choose $\bar v$ to be any positive critical
point of the energy functional $E$ \alert{associated with $\Phi_c$:}
\begin{equation}\label{eqE}
  E[v] := \int_\Omega \left( \frac 1 2 \vert\nabla_y v \vert^2+
    V(v,y)\right)\, dy , 
  \qquad v \in H^1(\Omega), \quad v \bigl|_{\partial \Omega_\pm}
  = 0.
\end{equation}
To each such $\bar v$ one can associate a minimizer of $\Phi_c$ in the
admissible class of functions that are bounded above by $\bar
v$. Then, under the assumption that the initial data approaches $\bar
v$ uniformly from below as $z \to -\infty$ one can make the conclusion
(under generic non-degeneracy assumptions) that the solution of the
initial value problem converges exponentially to the corresponding
minimizer. Thus, every front-like initial data in a more restricted
sense of connecting zero to a critical point $\bar v$ of $E$ converges
to the minimizer associated with that critical point. More precisely,
we have

\begin{Corollary}
  \label{c:main}
  Under hypotheses (H1)--(H3), (N1)--(N2), \alert{with the trial
    function $u$ in hypothesis (H3) satisfying $u \leq \bar v$, where
    $\bar v > 0$ is a critical point of $E$,} let $\bar u$ be the
  unique (up to translations) non-trivial minimizer of $\Phi_\cd$ over
  functions $u \in H^1_\cd(\Sigma)$ satisfying \eqref{bc} and $0 \leq
  u \leq \bar v$.  Let $u_0 \in C^0(\overline\Sigma) \cap
  W^{1,\infty}(\Sigma) \cap L^2_\cd(\Sigma)$ satisfy $0 \le u_0 \le
  \bar v$ and $u_0(\cdot, z) \to \bar v$ uniformly in $\Omega$ as $z
  \to -\infty$. Then the conclusion of Theorem \ref{t:main} holds.
\end{Corollary}

An important implication of Corollary \ref{c:main} is that $\bar v$
{\em selects} the attracting variational traveling wave solution in
the long time limit. This kind of conclusion was made by us earlier
for the propagation speed of the leading edge without the
non-degeneracy assumptions of the present paper \cite{mn:cms08}.

\alert{We note that} the problem of convergence to traveling waves for
solutions of \eqref{pde} has been widely considered in the
mathematical literature.  We refer to
\cite{fife79,volpert,roquejoffre97} and references therein, for a
general overview on the subject.  \alert{Specifically,} our result
should be compared with \cite[Theorem 3.7]{roquejoffre97} \alert{by
  Roquejoffre, where, in particular, convergence to variational
  traveling waves is proved (in our notation) for initial data that
  approach zero from above as $z \to +\infty$ and a non-degenerate
  local minimizer $\bar v > 0$ of $E$ from below as $z \to -\infty$.}

\alert{Roquejoffre makes a crucial} assumption that there
\alert{exists a variational traveling wave connecting $\bar v$ at $z =
  -\infty$ with zero at $z = +\infty$. In contrast, our results do not
  require existence of such a traveling wave. Instead, we require that
  the initial} \alert{data decay sufficiently rapidly to zero as $z
  \to +\infty$ and stay approximately above the local minimizer $v$ of
  $E$ corresponding to the limit at $z = -\infty$ for the special
  variational traveling wave $\bar u$ given by Theorem \ref{t:min} as
  $z \to -\infty$. Under this condition the solution of \eqref{pde} is
  attracted to a translate of $\bar u$ on compacts in the moving
  reference frame (see Theorem \ref{t:main} for a precise statement).
  We note that in the class of front-like initial data with
  sufficiently fast exponential decay considered by us global
  stability of a traveling wave connecting zero to $\bar v$ is a
  simple consequence of Corollary \ref{c:main}. Indeed, if there
  exists a variational traveling wave $u_c$ connecting zero to $\bar
  v$, then by Proposition \ref{p:max} we have $u_c = \bar u$, where
  $\bar u$ is as in Corollary \ref{c:main} (note that in this case
  hypotheses (H3) and (N2) are unnecessary). Thus, within the scope of
  \eqref{pde} and front-like initial data decaying sufficiently fast,
  our results are applicable to more general initial data than the
  ones considered in \cite{roquejoffre97} and, most importantly,
  provide a selection criterion for the limit front in terms of the
  asymptotic behavior of the initial data as $z \to -\infty$.}
We also point out that our assumptions concerning the nonlinearity $f$
(see (H1)--(H3) below) are quite general compared to the assumptions
usually made in the literature
\cite{berestycki92,vega93,roquejoffre97}. \alert{In particular, these
  assumptions can be readily verified in practice (for examples see
  \cite{lmn:cpam04,lmn:arma08,mn:cvar08}).}

\section{Preliminaries}\label{secvar}

Throughout this paper we assume $\Omega$ to be a bounded domain
(connected open set, not necessarily simply connected) with a boundary
of class $C^2$. We start by listing the assumptions on the
nonlinearity $f$ which we need in Theorem \ref{t:main}.  The function
$f: [0,1]\times\overline \Om \to \mathbb R$ satisfies:
\begin{eqnarray*} 
  {\bf (H1)} && \quad f(0,y) = 0 \qquad f(1, y) \leq
  0 \qquad \textrm{for all } y \in \Omega,
  \\
  {\bf (H2)} && \quad f\in C^{0,\gamma}([0,1]\times\overline\Om) 
  \qquad f_u = \frac{\partial f}{\partial u}\in
  C^{0,\gamma}([0,1]\times\overline \Om) \ \textrm{for some $\gamma \in (0,1)$}. 
\end{eqnarray*}

Hypotheses (H1) and (H2) are needed to guarantee, in particular,
existence and basic regularity properties of solutions of \eqref{pde}.
Indeed, from \cite[Proposition 5.1]{mn:cms08} and \cite[Chapter
7]{lunardi} we have the following

\begin{Proposition} 
  \label{p:ini} 
  Under assumptions (H1) and (H2), let $u_0 \in C^0(\overline\Sigma)
  \cap W^{1,\infty}(\Sigma)$. Let also $u_0$ satisfy the boundary
  conditions (\ref{bc}) and assume $u_0(x) \in [0, 1]$ for all $x \in
  \Sigma$. Then there exists a unique solution (using notation of
  \cite{evanspde})
  \[
  u \in C^2_1( \Sigma \times (0, \infty)) \cap C^0(\overline \Sigma
  \times [0, +\infty))
  \]
  of (\ref{pde}) with boundary conditions (\ref{bc}) and initial
  condition $u(\cdot, 0) = u_0$, which satisfies $0 \leq u \leq 1$ and
  $||\nabla u||_{C_b(\overline\Sigma \times (0, +\infty))} < \infty$.
  Moreover, letting $\Sigma_M:=\Omega\times [M,M+1]$ for all $M\in\R$,
  we have
  \begin{equation}\label{stimawp}
    \|u(\cdot, t)\|_{W^{2,p}(\Sigma_M)} \le C(t_0,p)
    \qquad {\rm \ for\ all\ } t \geq t_0>0,\ p>1. 
  \end{equation}
  Finally, if $u_0\in L^2_c(\Sigma)$ for some $c > 0$, we also have
  \begin{equation}\label{stimahc}
    u\in C^\alpha((0, +\infty); H^2_c) \cap
    C^{1,\alpha}((0,+\infty);L^2_c(\Sigma)) \qquad    
    {\rm for\ all\ }\alpha \in (0,1),
  \end{equation}
  and
 \begin{equation}\label{h3c}
   \alert{u_t\in C((0,+\infty);H^1_c(\Sigma)).}
  \end{equation}
\end{Proposition}

We now turn to the assumption which is both necessary and sufficient
for the existence of the special variational traveling wave solution
considered in this paper  \cite{lmn:cpam04,mn:cms08}. 

\medskip

\begin{description}
\item[\hspace{4mm} (H3)] There exist $c > 0$, satisfying $c^2 + 4
  \nu_0 > 0$, where
  \begin{eqnarray}
    \label{eq:nu0}
    \nu_0 = \min_{\substack{\psi \in H^1(\Omega) \\ \psi|_{\partial
          \Omega_\pm} = 0}} \frac{\int_\Omega  ( |\nabla_y \psi|^2 -
      f_u(0, y) \psi^2 ) \, dy}{\int_\Omega  \psi^2 \, dy},
  \end{eqnarray}
  \hspace{2.5mm} and $u \in H^1_c(\Sigma)$, such that $\Phi_c[u]\leq 0$
  and $u \not \equiv 0$.
\end{description}

\medskip

\begin{Remark}\label{r:eneg}\rm 
  As was shown in \cite{mn:cms08}, in the case $\nu_0 \geq 0$ the
  hypothesis (H3) is equivalent to the condition
  \begin{eqnarray}
    \label{eq:6}
    \inf_{\substack{v \in H^1(\Omega) \\ v|_{\partial
          \Omega_\pm} = 0}} E[v] < 0.
  \end{eqnarray}
\end{Remark}

Under the above assumptions, we can state the existence result
concerning the variational traveling wave which is the minimizer of
$\Phi_c$ with a suitably fixed translation.
\begin{Theorem}
  \label{t:min}
  Under hypotheses (H1)--(H3), there exists a unique value of $c^\dag
  \geq c$, where $c$ is defined by hypothesis (H3), and a unique
  function $\bar u \in C^2(\Sigma) \cap C^1(\overline \Sigma)$, $\bar
  u \not\equiv 0$, such that $( c^\dag, \bar u)$ solve (\ref{trav})
  and (\ref{bc}), and $\bar u$ satisfies $||\bar u(\cdot,
  0)||_{L^\infty(\Omega)} = \tfrac12 \sup_{z \in \mathbb R} ||\bar
  u(\cdot, z)||_{L^\infty(\Omega)}$ and minimizes $\Phi_\cd$ in
  $H^1_\cd(\Sigma)$.  Moreover $\bar u \in H^2_\cd(\Sigma) \cap
  W^{1,\infty}(\Sigma)$, \alert{$\bar u_z \in H^2_\cd(\Sigma)$}, $\bar
  u_z < 0$ in $\Sigma$, and
  \begin{equation} 
  \label{eqbar}
  \lim_{z\to +\infty}\bar u(\cdot, z)= 0, \quad \lim_{z\to
    -\infty}\bar u(\cdot, z)= v \qquad {\rm\ in \ 
  } C^1(\overline\Omega),
  \end{equation}
  where $v: \Omega \to \mathbb R$ is a local minimizer of $E$ defined
  in \eqref{eqE}, with $E[v] < 0$.
\end{Theorem}
\alert{

\medskip

\noindent For the proof see \cite[Theorem 3.3]{mn:cms08} and
\cite[Proposition 3.3(ii)]{lmn:arma08} (the latter argument also
applies to $\bar u_z$ by differentiating \eqref{trav} in $z$).

\medskip

Let us point out that the minimizer of Theorem \ref{t:min}
is in some sense the ``maximal'' variational traveling wave
solution. More precisely, we have the following result:

\begin{Proposition}
  \label{p:max}
  Let hypotheses (H1)--(H3) be satisfied, and let $(c, u)$ solve
  \eqref{trav} and \eqref{bc}, with $c > 0$, $u \in H^1_c(\Sigma)$ and
  $0 < u < 1$. Then, if $v$, $c^\dag$, $\bar u$ are as in Theorem
  \ref{t:min}, and
  \begin{eqnarray}
    \label{eq:vp}
    \liminf_{z \to -\infty} u(\cdot, z) \geq v \qquad \qquad
    \text{uniformly in } \Omega,
  \end{eqnarray}
  we have $c = c^\dag$ and $u = \TT_R \bar u$, for some $R \in \mathbb
  R$. In particular, the inequality in \eqref{eq:vp} is, in fact,
  equality.
\end{Proposition}

\begin{proof}
  First note that we cannot have $c > c^\dag$. Indeed, if this
  inequality were true, by \cite[Proposition 3.5]{lmn:arma08} the pair
  $(c, u)$ can be taken as a trial function in hypothesis (H3),
  contradicting the conclusion of Theorem \ref{t:min} that $c^\dag
  \geq c$. On the other hand, it is easy to see that $c < c^\dag$ is
  also impossible. Indeed, arguing as in the proof of
  \cite[Proposition 5.5]{mn:cms08}, for any $c' \in (c, c^\dag)$ there
  exists a non-trivial minimizer $\bar u_{c'}$ of $\Phi_{c'}$ in the
  class of functions in $H^1_{c'}(\Sigma)$ which stay below $v$ and
  vanish outside $\Sigma_R = \Omega_\sigma \times (-R, R)$, with
  $\Omega_\sigma = \{y \in \Omega : \mathrm{dist}(y, \partial
  \Omega_\pm) > \sigma\}$, where $R > 0$ is large enough and $\sigma >
  0$ is small enough.\footnote{This choice of $\Sigma_R$ also corrects
    a minor inaccuracy in the proof of \cite[Proposition
    5.5]{mn:cms08}.}  Furthermore, $\bar u_{c'}$ is a classical
  solution of \eqref{trav} with $c = c'$ in $\Sigma_R$ and $\bar
  u_{c'} \leq \max(0, v - \eps)$ for some $\eps > 0$. Therefore, by
  \eqref{eq:vp} the function $\TT_{R'} \bar u_{c'} < u$ in $\Sigma$
  for some $R' \in \mathbb R$ sufficiently large negative, and by
  parabolic comparison principle \cite{protter} we have $\bar u_{c'} <
  \TT_{(c - c') t - R'} u$ for all $t > 0$. However, the latter is
  impossible, since the right-hand side of this inequality converges
  to zero in $H^1_c(\Sigma)$. Thus $c = c^\dag$, hence \alert{$u$ is a
    minimizer by \cite[Proposition 3.5]{lmn:arma08}, and the result
    follows from \cite[Theorem 3.3(v)]{mn:cms08}.}
\end{proof}

We note that, in particular, the result in Proposition \ref{p:max}
allows to extend the statements about monotonicity and uniqueness of
traveling waves established in the classical work of Berestycki and
Nirenberg \cite{berestycki92} for \eqref{pde} and \eqref{bc} (see also
\cite{vega93,vega93jmma}), in the class of variational traveling
waves, under only an assumption that the traveling wave approaches a
limit from below as $z \to -\infty$, zero from above as $z \to
+\infty$, and is sandwiched between these two limits. Indeed, suppose
$(c, u)$ is such a traveling wave, with $u(\cdot, z) \to \bar v$ as $z
\to -\infty$, with $0 < \bar v \leq 1$. Then by the argument of
\cite[Proposition 6.6]{lmn:arma08} $v$ is a critical point of $E$, and
by \cite[Proposition 3.5]{mn:cms08} we have $c^2 + 4 \nu_0 > 0$. So by
Proposition \ref{p:max} this traveling wave is a non-trivial minimizer
of $\Phi_c$ in $H^1_c(\Sigma)$ over all positive functions bounded
above by $\bar v$, and the result follows from \cite[Theorem
3.3]{mn:cms08}. In particular, we do not require any non-degeneracy
assumptions for the limits of $u(\cdot, z)$ as $z \to \pm \infty$, as
is done \cite{berestycki92}. Thus, we have:

\begin{Corollary}
  Under hypotheses (H1) and (H2), let $c > 0$, and let $u \in
  H^1_c(\Sigma)$ be a solution of \eqref{trav} and \eqref{bc},
  satisfying $u(\cdot, z) \to \bar v$ uniformly in $\Omega$ as $z \to
  -\infty$, where $0 < \bar v \leq 1$ and $0 < u < \bar v$. Then $c^2
  + 4 \nu_0 > 0$, the value of $c$ is unique, $u_z < 0$, and $u$ is
  unique up to translations.
\end{Corollary}
}

We now list two additional technical assumptions (see also
\cite{roquejoffre97}), which are generically satisfied and are needed
to prove global exponential stability of the minimizers of $\Phi_c$
for initial data bounded below by $v$ as $z \to -\infty$.

\medskip

\begin{description}
\item[\hspace{4mm} (N1)] For $v$ as in Theorem \ref{t:min} we have
\begin{eqnarray}
  \label{eq:nut}
  \tilde \nu_0 = \min_{\substack{\psi \in H^1(\Omega) \\
      \psi|_{\partial \Omega_\pm} = 0}} \frac{\int_\Omega  ( |\nabla_y
    \psi|^2 - f_u(v, y) \psi^2 ) \, dy}{\int_\Omega  \psi^2 \, dy} >0. 
\end{eqnarray}
\item[\hspace{4mm} (N2)] For $v < 1$ as in Theorem \ref{t:min} there
  is no solution $(c^\dag,\bar u)$ of \eqref{trav} and \eqref{bc},
  with $c^\dag$ as in Theorem \ref{t:min}, such that $v< \bar u< 1$ on
  $\Sigma$.
\end{description}

\medskip

\alert{Conditions (N1) and (N2) are generic in the sense that the set
  of nonlinearities $f$ such that (N1) or (N2) do not hold is a meager
  subset of all $f$'s obeying (H1)-(H3), in the natural topology (for
  similar notions related to perturbations of $\Omega$ see
  \cite{henry05}).  Indeed,} condition (N1) is generic, \alert{since}
by the results of \cite{mn:cms08} we have $\tilde \nu_0 \geq 0$, so
that (N1) only excludes the degenerate case of $\tilde \nu_0 =
0$. Similarly, condition (N2) excludes the non-generic possibility of
existence of a traveling front invading $v$ from above with the same
speed $c^\dag$ as the front invading zero by $v$. To see that the only
non-trivial alternative would be to have a front invading $v$ with
lower speed, consider the following variational problem. Given $c>0$
and $h\in H^1_c(\Sigma)$ satisfying \eqref{bc}, let
\begin{eqnarray}
  \label{Psic}
  \Psi^v_c[h] := \int_{\Sigma} e^{cz} \left( \frac{ |\nabla h|^2}{2} +
    V(v+h,y)-V(v, y)-V'(v,y)\,h \right) \, dx, 
\end{eqnarray}
where we used the notation $V'(s, y) := \partial V(s, y) / \partial
s$.  Notice that, if $\bar h$ is a critical point of
$\Psi_{c^\dag}^v$, then $\bar u=v+\bar h$ is a solution of
\eqref{trav} and \eqref{bc}. We set
\begin{eqnarray}
  \label{cdagv}
  c^\dag_v := \inf\Big\{ c> 0:\ \Psi^v_c[h] \geq 0 \textup{ for all
  } h \geq 0 \Big\}. 
\end{eqnarray}
Then the following result concerning $c^\dag_v$ holds.

\begin{Lemma}\label{lemN}
  The functional $\Psi^v_c$ is weakly sequentially lower
  semicontinuous and coercive in $H^1_c(\Sigma)$ for all $c>c^\dag_v$,
  and $c^\dag_v\le c^\dag$, where $\cd$ is as in Theorem \ref{t:min}.
  Moreover, under hypothesis (N2) we have $c^\dag_v< c^\dag$.
\end{Lemma}

\noindent In other words, under hypothesis (N2) it is only possible to
have such a system of stacked waves \cite{volpert} invading zero, that
the front connecting zero with $v$ moves faster than the front
invading $v$ from above.

\emph{Proof of Lemma \ref{lemN}.} First of all, reasoning as in
\cite[Proposition 5.5]{lmn:arma08} and using the fact that $\tilde
\nu_0 \geq 0$, where $\tilde \nu_0$ is defined in \eqref{eq:nut}
\cite[Theorem 3.3(iv)]{mn:cms08}, one can see that $\Psi^v_c$ is
weakly lower semicontinuous in $H^1_c(\Sigma)$ for all $c>0$, so the
results of \cite{mn:cms08} apply to $\Psi_c^v$. Moreover, reasoning as
in the proof of \cite[Proposition 6.9]{lmn:arma08}, we also get that
$\Psi^v_{c}$ is coercive in $H^1_c(\Sigma)$ for all $c>c^\dag_v$.

  Let us now prove that $c^\dag_v\le c^\dag$.  Assume by contradiction
  that there exists $w \geq v$, such that $\Psi_{c^\dag}^v[w-v]<0$.
  Slightly perturbing $w$, we can ensure that $w = v$ for $z \ge z_0$,
  with $z_0\in \R$ big enough.  Let $\bar u$ be the minimizer of
  $\Phi_{c^\dag}$ given by Theorem \ref{t:min}, and let $\eps\le
  -\Psi_{c^\dag}^v[w-v]/2$. Since $\bar u(\cdot,z)\to v$ in
  $H^1(\Omega)$ as $z \to -\infty$, up to a suitable translation we
  can perturb $\bar u$ into a function $\tilde u\in
  H^1_{c^\dag}(\Sigma)$ such that $\tilde u = v$ for $z \le z_0$ and
  $\Phi_{c^\dag}[\tilde u]\le \eps$.  Define $\hat u \in
  H^1_\cd(\Sigma)$ as
  \[
  \hat u(y,z):= \left\{
    \begin{array}{ll} w(y,z) & {\rm if}\ z\le z_0
      \\
      \tilde u(y,z) & {\rm if}\ z> z_0.
    \end{array}
  \right.
  \]
  Letting $h=w-v \in H^1_\cd(\Sigma)$ and satisfying \eqref{bc}, after
  an integration by parts and using the Euler-Lagrange equation for
  $E$ satisfied by $v$, we get
  \begin{eqnarray*}
    \Phi_{c^\dag}[{\hat u}] &=& 
    \int_{-\infty}^{z_0} \int_{\Om} e^{c^\dag z}
    \left( \frac{ |\nabla 
        (v+h)|^2}{2} + V(v+h,y) \right) \, dy \, dz
    \\
    &&+ \int_{z_0}^{+\infty} \int_{\Om} e^{c^\dag z} \left( \frac{
        |\nabla \tilde u|^2}{2} + V(\tilde u,y) \right) \, dy \, dz
    \\
    &=& \int_{-\infty}^{z_0} \int_{\Om} e^{c^\dag z} 
    \left( \frac{ |\nabla h|^2}{2} + V(v+h,y)-V(v,y)-V'(v,y)h \right)
    \, dy \, dz
    \\
    &&+ \int_{-\infty}^{z_0} \int_{\Om} e^{c^\dag z} 
    \left( \frac{ |\nabla_y v|^2}{2} +V(v,y)\right) \, dy \, dz 
    \\
    &&+ \int_{z_0}^{+\infty} \int_{\Om} e^{c^\dag z} \left( \frac{
        |\nabla \tilde u|^2}{2} + V(\tilde u,y) \right) \, dy \, dz 
    \\
    &=& \Psi_{c^\dag}^v[h] + \Phi_{c^\dag}[\tilde u] \le
    \frac{\Psi_{c^\dag}^v[h]}{2}<0\,
  \end{eqnarray*}
  which contradicts the minimizing property $\Phi_\cd[\bar u] = 0$ of
  $\bar u$ \cite[Proposition 3.2]{lmn:arma08}.

  To conclude the proof, it remains to prove that $c^\dag_v< c^\dag$
  under hypothesis (N2).  If $c^\dag_v= c^\dag$, then for every $c\in
  (0,c^\dag)$, there exists a function $h_c \not\equiv 0$, such that
  $\Psi_c^v[h_c] < 0$. Hence, the analog of hypothesis (H3) holds for
  $\Psi_c^v$, and, therefore, there exists a non-trivial minimizer
  $\bar h$ of $\Psi_{\hat c}^v$ for some $\hat c \geq \cd$. On the
  other hand, by the argument of \cite[Proposition 5.5]{mn:cms08}, we
  have $\hat c \leq c^\dag_v$. So $\hat c = \cd$, and since $\hat u =
  v + \bar h > v$ is a solution of \eqref{trav} and \eqref{bc} with $c
  = \hat c$, this violates assumption (N2).
  \qquad \endproof

Finally, we note that if either (N1) or (N2) are violated, one would
not expect exponential stability of $\bar u$ in the reference frame
moving with speed $c^\dag$ any more. Therefore, in some sense these
conditions are also necessary for the results obtained by us.

\section{Local stability in $L^2_\cd(\Sigma)$}\label{secsta}

In this section we prove stability of the variational traveling wave
$\bar u$ minimizing $\Phi_\cd$ in the reference frame moving with
speed $\cd$ up to perturbations which are small in the $L^2_\cd$-norm
and stay approximately above $v$ behind the front.

\begin{Theorem}\label{t:wnstab}
  Assume hypotheses (H1)--(H3) and (N1)--(N2) hold, and let $\bar u$
  and $\cd$ be as in Theorem \ref{t:min}. Then there exist $\alpha >
  0$ and $\sigma > 0$, such that for every $u_0$ as in Theorem
  \ref{t:main} and for every $\omega > 0$ there exists $\eps > 0$,
  such that if
  \begin{eqnarray}\label{eqe}
    || u_0 - \bar u||_{L^2_\cd(\Sigma)}\leq \eps,
  \end{eqnarray}
  the solution $u(x, t)$ of
  \begin{equation}\label{pdec}
    u_t = \Delta u + \cd u_z + f(u, y),
  \end{equation}
  with boundary conditions in \eqref{bc} and $u(x, 0) = u_0(x)$
  satisfies
  \begin{eqnarray}
    \label{wnstab}
    ||u(\cdot, t) - \TT_{R_\infty} \bar u||_{L^2_\cd(\Sigma)} \leq
    \omega e^{-\sigma t}, \qquad |R_\infty| \leq \omega,
  \end{eqnarray}
  for some $R_\infty \in \mathbb R$.
\end{Theorem}

We note that our approach differs somewhat from the conventional
approach to the studies of front stability
\cite{sattinger76,berestycki92arma,roquejoffre92,roquejoffre94} in the
way we treat translations along the cylinder axis. We track the front
position by minimizing the $L^2_\cd$-distance between the solution of
\eqref{pdec} and a translate of $\bar u$. As a consequence, the
deviation between the solution and the closest translate of $\bar u$
is automatically orthogonal to the null-space of the linearization
operator, allowing to readily establish the exponential decay of the
$L^2_\cd$-distance. Thus, our method is more variational in
nature. Let us also point out that, in contrast to the usual approach,
our initial data do not need to be close to $\bar u$ in $L^\infty$ in
the whole cylinder, they may be significantly larger than $\bar u$ at
large negative $z$.

Throughout the rest of this section, hypotheses (H1)--(H3) and
(N1)--(N2) are assumed to hold, and $\cd$, $\bar u$, $v$ always refer
to the minimizer in Theorem \ref{t:min}. We begin with the following
basic lemma concerning the linearization around $\bar u$.

\begin{Lemma}\label{lemlin}
  There exists $K > 0$, such that
  \begin{eqnarray}
    \label{l2}
    \int_\Sigma e^{\cd z} \left( |\nabla w|^2 - f_u(\bar u, y) w^2
    \right) dx \geq K \int_\Sigma e^{\cd z} w^2 dx,
  \end{eqnarray}
  for all $w \in H^1_\cd(\Sigma)$ satisfying $\int_\Sigma e^{\cd z} w
  \bar u_z dx = 0$.
  \end{Lemma}

\begin{proof}
  First of all, observe that by choosing $R_1$ and $R_2$ sufficiently
  large, we have
  \begin{eqnarray}
    {\int_{-\infty}^{-R_1} \int_\Omega e^{\cd z} \left( |\nabla w|^2 -
        f_u(\bar u, y) w^2 \right) dy \, dz \over  \int_{-\infty}^{-R_1}
      \int_\Omega e^{\cd z} w^2 \, dy \, dz} \geq K_1 > 0, \label{K1} \\ 
    {\int_{R_2}^{+\infty} \int_\Omega e^{\cd z} \left( |\nabla w|^2 -
        f_u(\bar u, y) w^2 \right) dy \, dz \over  \int_{R_2}^{+\infty}
      \int_\Omega e^{\cd z} w^2 \, dy \, dz} \geq K_2 > 0,\label{K2}
  \end{eqnarray}
  for all $w \in H^1_\cd(\Sigma)$. Indeed, if $z$ is large enough
  negative, then $\bar u(\cdot, z)$ is sufficiently close in
  $L^\infty(\Omega)$ to $v$. Hence, (\ref{K1}) holds in view of
  (\ref{eq:nut}). On the other hand, by the estimate of \cite[Lemma
  5.1]{lmn:arma08}, we have
  \begin{eqnarray*}
    &&\int_{R_2}^{+\infty} \int_\Omega e^{\cd z} \left( |\nabla w|^2 -
      f_u(0, y) w^2 \right) dy \, dz 
    \\ 
    &&\geq \int_{R_2}^{+\infty} \int_\Omega e^{\cd z} \left(
      {\cd^2 \over 4}  w^2 + |\nabla_y w|^2 + f_u(0, y) w^2 \right) dy
    \, dz 
    \\ 
    &&\geq \left( {\cd^2 \over 4} + \nu_0 \right)
    \int_{R_2}^{+\infty}  \int_\Omega e^{\cd z} w^2 \, dy \, dz.
  \end{eqnarray*}
  So, by hypothesis (H3) and \eqref{eqbar}, the inequality in
  (\ref{K2}) holds for some $K_2 > 0$ and $R_2$ large enough.

  Let us now show that the inequality in \eqref{l2} holds with $K = 0$
  for all $w \in H^1_\cd(\Sigma)$ and that equality holds if and only
  if $w$ is a multiple of $\bar u_z$ (the proof essentially follows
  the ideas of concentration compactness principle in the case of
  exponentially weighted Sobolev spaces \cite{lions84a,struwe} and
  relies on the maximum principle). Indeed, denote by $H[w]$ the
  left-hand side of \eqref{l2} and let $(w_n)$ be a minimizing
  sequence for $H$ subject to the constraint
  $||w_n||_{L^2_\cd(\Sigma)} = 1$. By coercivity of $H$ on the
  constraint, ensured by hypothesis (H2), we have $w_n \rightharpoonup
  w_0$ in $H^1_\cd(\Sigma)$. In fact, $w_0 \not = 0$, since otherwise
  $w_n \to 0$ in $L^2_\mathrm{loc}(\Sigma)$, and so
  $\int_{-\infty}^{-R_1} \int_\Omega e^{\cd z} w_n^2 dx +
  \int_{R_2}^{+\infty} \int_\Omega e^{\cd z} w_n^2 dx \geq 1 - \eps$
  for any $\eps > 0$ and large enough $n$. Therefore, by \eqref{K1}
  and \eqref{K2} we would have $H[w_n] \geq K > 0$. However, this
  contradicts the fact (first pointed out in \cite{barenblatt57}) that
  $\bar u_z$ is an eigenfunction associated with zero eigenvalue of
  the linearization of \eqref{trav} around $\bar u$ (related to the
  translational symmetry in the $z$-direction
  \cite{barenblatt57,sattinger76,berestycki92arma,roquejoffre92}),
  which can be seen by differentiating \eqref{trav} with respect to
  $z$ and noting that $\bar u_z \in \alert{H^2_\cd(\Sigma)}$ by
  Theorem \ref{t:min}.

  In view of lower semicontinuity of $H$ with respect to the weak
  convergence in $H^1_\cd(\Sigma)$, which follows from
  \cite[Proposition 5.5]{lmn:arma08}, hypothesis (H3) and Theorem
  \ref{t:min}, we have $\displaystyle H[w_0] \leq \liminf_{n \to
    \infty} H[w_n] \leq ||\bar u_z||_{L^2_\cd(\Sigma)}^{-2} H[\bar
  u_z] = 0$. Then, since $w_0 \not=0$, the function $\bar w =
  ||w_0||_{L^2_\cd(\Sigma)}^{-1} |w_0| \geq 0$ is a minimizer of the
  considered constrained minimization problem. In fact, $H[\bar w] =
  0$, since otherwise $\bar w$ must be orthogonal to $\bar u_z$, which
  is impossible due to the fact that $\bar u_z < 0$ by Theorem
  \ref{t:min}. So $H[w] \geq 0$ for all $w \in
  H^1_\cd(\Sigma)$. Moreover, $H[w] = 0$ implies that $w$ is a
  multiple of $\bar u_z$ (compare also with
  \cite{berestycki92arma,roquejoffre92,roquejoffre97}). If not, there
  exists a minimizer $w'$ which is orthogonal to $\bar u_z$ in
  $L^2_\cd(\Sigma)$ and, therefore, changes sign. But $|w'|$ is also a
  minimizer, hence both $w'$ and $|w'|$ satisfy the linearized version
  of \eqref{trav} in the classical sense, thanks to hypothesis (H2)
  and Theorem \ref{t:min}. So by strong maximum principle $|w'| = 0$,
  leading to a contradiction.

  To complete the proof of the lemma, suppose, to the contrary of its
  statement, there exists a sequence $(w_n)$ with the properties that
  $||w_n||_{L^2_\cd(\Sigma)} = 1$, $\int_\Sigma e^{\cd z} \bar u_z w_n
  \, dx = 0$ and $H[w_n] \to 0$ as $n \to \infty$. Hence $w_n$ is a
  minimizing sequence and converges to a non-trivial multiple of $\bar
  u_z$ weakly in $H^1_\cd(\Sigma)$. But this contradicts the
  orthogonality of $w_n$ to $\bar u_z$, which is preserved in the
  limit as $n \to \infty$.  
\end{proof}

\alert{Let us note that one may naturally think that the result of
  Lemma \ref{lemlin} may be used to show that the minimizer $\bar u$
  is, in fact, a strict minimizer of $\Phi_\cd$ on a suitable subset
  of $H^1_\cd(\Sigma)$. This, however, proves difficult, since the
  functional $\Phi_c[u]$ is not {\em a priori} twice continuously
  differentiable in $H^1_c(\Sigma)$. We will get back to this question
  after Proposition \ref{proh} below.}

Our next result shows that, if a solution to \eqref{pdec} with initial
datum satisfying \eqref{inidat} with $\alpha$ sufficiently small is
close enough to a suitable translate of $\bar u$ in $L^2_\cd(\Sigma)$,
then it is also close in $L^\infty$ on some growing portion of
$\Sigma$, provided that $\eps$ is small enough. More precisely, let
$R: [0, \infty) \to \mathbb R$. For a given $\delta > 0$, we define
$z_\delta: [0, \infty) \to \overline{\mathbb R}$ as
\begin{equation}\label{zeta}
  z_\delta(t) := \sup\big\{ z\in\R:\, \| u(\cdot, z, t) - \bar
  u(\cdot, z - R(t))\|_{L^\infty(\Omega)} > \delta\big\} \qquad
  \forall t\ge 0.
\end{equation}
Then, the following result holds true.
\begin{Proposition}\label{p:w2close}
  There exists $b>0$, such that for every $\delta>0$ sufficiently
  small there exist $\alpha = \alpha(\delta) > 0$, $a = a(\delta) >
  0$, $\bar z_0 = \bar z_0(\delta, u_0) \in \mathbb R$ and $\eta =
  \eta(\delta, u_0) > 0$, such that for every $z_0 \leq \bar z_0$
  there exists $\eps = \eps(\delta, z_0) > 0$ such that for all $T >
  0$
  \begin{equation}
    \label{w0infty}
    z_\delta(t)\le z_0 + a - bt \qquad \forall t\in [0,T],
  \end{equation}
  whenever
  \begin{equation}\label{equesto}
    |R(t)| \leq \delta ~~\mathrm{and}~~ \| u(\cdot, t) -
    \TT_{R(t)} \bar u \|_{L^2_\cd(\Sigma)} 
    \leq \eta \qquad \forall t\in [0,T],
  \end{equation} 
  where $u_0$, $u$, $\alpha$, $\eps$ are as in Theorem \ref{t:wnstab}.
\end{Proposition}

\begin{proof}
  By (\ref{equesto}) and the uniform Lipschitz continuity of $u(\cdot,
  t)$ in $\Sigma$, reasoning as in the proof of \cite[Proposition
  3.3(iii)]{lmn:arma08} we have the following $L^\infty$-estimate:
  \begin{eqnarray}
    \label{linfcr}
    || u(\cdot, t) - \TT_{R(t)} \bar u ||^{n+2}_{C_b(\overline \Omega 
      \times [z_0, +\infty))}  \leq C \eta^2 e^{-\cd z_0},
  \end{eqnarray}
  for any $z_0 \in \mathbb R$, any $t \in [0, T]$ and some $C > 0$
  depending on $||\nabla u||_{C_b(\overline\Sigma \times (0,
    +\infty))}$ (see Proposition \ref{p:ini}).  On the other hand, by
  Theorem \ref{t:min} for any $\alpha > 0$ there exists $\bar z_0 \in
  \mathbb R$, such that
  \begin{equation}\label{zv}
    || \bar u(\cdot, z-R(t))-v||_{C^0(\overline \Omega)} \leq
    \alpha \qquad \forall z \leq \bar z_0, ~ \forall t \in [0, T].
  \end{equation}
  Recalling \eqref{inidat} and possibly reducing $\bar z_0$,
  we can also assume that 
  \begin{equation}
    \label{uzv}
    u_0(\cdot, z)\ge v - 2 \alpha \qquad \forall z\le
    \bar z_0. 
  \end{equation}
  Now, choosing $\eta > 0$ sufficiently small, the right-hand side of
  (\ref{linfcr}) can be bounded by $\alpha^{n+2}$ at $z_0 = \bar z_0$,
  so we have
  \begin{eqnarray}
    \label{z0p}
    ||u(\cdot, t) - \TT_{R(t)} \bar u||_{C_b(\overline \Omega \times
      [\bar z_0, +\infty))} \leq \alpha \qquad \forall t \in [0, T]. 
  \end{eqnarray}
  Therefore,
  \begin{equation}
    \label{linf0}
    ||u(\cdot, z, t) - \bar u(\cdot, z -
    R(t))||_{C^0(\overline\Omega)} \leq \delta  \qquad \forall z
    \geq \bar z_0, ~\forall t \in [0,T],
  \end{equation}
  as long as $\alpha \leq \delta$, so that $z_\delta(t) \leq \bar z_0$
  for all $t \in [0, T]$.

  It remains to show that the inequality in \eqref{linf0} also holds
  for $z\in [z_0+a-bt, \bar z_0]$, for some positive $a$ and $b$, for
  small enough $\alpha$ and $\eps$.  We proceed by constructing
  explicit upper and lower barriers for (\ref{pdec}) in $\Omega \times
  (-\infty ,\bar z_0] \times [0, T]$.
  
  \medskip
  
  \noindent{\it Subsolution.}
  First, consider the case of $\partial\Omega_\pm = \varnothing$,
  i.e., pure Neumann boundary conditions in \eqref{bc}. Then, it is
  straightforward to verify that by hypotheses (H1)--(H2) and (N1) the
  function $v_\delta^- = v - C \delta \tilde \psi_0$, where
  $\tilde\psi_0 > 0$ is an eigenfunction associated with $\tilde\nu_0$
  in \eqref{eq:nut} and $\displaystyle C =
  ||\tilde\psi_0||_{C^0(\overline\Omega)}^{-1}$, is the desired
  subsolution, provided that $\delta$ is sufficiently small.

  The construction is more delicate in the presence of Dirichlet
  boundary conditions, since we do not wish to put any restrictions on
  the derivative of the initial data near the Dirichlet portion of the
  boundary. So, let us now assume that $\partial \Omega_\pm \not=
  \varnothing$, implying, in particular, that $v < 1$ in
  $\overline\Omega$. We construct a subsolution in the form of a
  non-negative local minimizer of $E$ that lies sufficiently close to
  and below $v$, and vanishes identically within some small distance
  to $\partial \Omega_\pm$.

  We proceed in the usual way by introducing the modified energy
  $\tilde E$, given by \eqref{eqE} in which $V$ is replaced by $\tilde
  V(u, y) = - \int_0^u \tilde f(s, y) ds$, where $\tilde f$ is
  obtained from $f$ by the odd extension for $u < 0$ and the $C^1$
  linear extrapolation for $|u - v| > \delta$, for some fixed $0 <
  \delta \ll 1$ and each $y \in \Omega$. We note that $\tilde f(u, y)
  = f(u, y)$ whenever $|u - v| \leq \delta$ and $u \geq 0$. Now, by
  hypotheses (H2) and (N1) the energy $\tilde E$ is strictly convex
  for all functions vanishing on $\partial \Omega_\pm$ and, hence,
  admits a unique minimizer $v_\delta^- \in H^1(\Omega)$ in the class
  of functions vanishing outside $\Omega_\sigma = \{ y \in \Omega:
  \mathrm{dist}(y, \partial \Omega_\pm) > \sigma\}$, with $\sigma > 0$
  sufficiently small. Moreover, we have $|v_\delta^- - v| = O(\sigma)$
  in $\Omega_\sigma$. Indeed, testing $\tilde E$ with $\tilde v =
  \max(0, v - C \sigma)$ for $C > 0$ so large that $\tilde v \equiv 0$
  in $\Omega \backslash \Omega_\sigma$ and using coercivity of $\tilde
  E$ and the fact that $v$ satisfies the Euler-Lagrange equation for
  $\tilde E$ in the whole of $\Omega$, we obtain that $||v_\delta^- -
  v||_{L^2(\Omega)} = O(\sigma)$. Therefore, by elliptic regularity
  theory \cite{gilbarg} \alert{and possibly reducing $\sigma$, we have
    $||v_\delta^- - v||_{L^\infty(\Omega)} = O(\sigma)\le \delta$,}
  and so $v_\delta^-$ satisfies the Euler-Lagrange equation for the
  original energy $E$ whenever $v_\delta^- > 0$.

  In fact, $v_\delta^- \geq 0$ in $\Omega$ and is strictly positive in
  $\Omega_\sigma$. Indeed, by its definition the function $\tilde V(u,
  \cdot)$ is even, whenever $|u - v| \leq \delta$. Hence, if
  $v_\delta^-$ is a minimizer satisfying the latter inequality, so is
  $|v_\delta^-|$. But by uniqueness the two must be equal. On the
  other hand, this implies that $v_\delta^-$ is a critical point of
  the original energy $E$. Therefore, by strong maximum principle we
  have $v_\delta^- > 0$ in $\Omega_\sigma$. Similarly, we must have
  $v_\delta^- < v$ in $\Omega$, since $\bar v = v + a \tilde \psi_0$
  is a strict supersolution for any $0 < a \ll 1$ and, therefore,
  cannot touch $v_\delta^-$ from above. Thus, we constructed a
  function $v_\delta^-$ which is a non-negative subsolution of the
  Euler-Lagrange equation for $E$, and $0 \leq v_\delta^- \leq v$. In
  particular, by construction
  \begin{equation}
    \label{eq:vdel}
    v - \delta \leq v_\delta^- \leq \max( 0, v - 2 \alpha),
  \end{equation}
  in $\Omega$, for $\alpha$ sufficiently small, depending only on
  $\delta$. Finally, extending this function to $\Sigma \times [0, T]$
  by defining $u^-(y, z, t) := v_\delta^-(y)$, we obtain a subsolution
  on the desired domain.
 
  \medskip
  
  \noindent{\it Supersolution.} Let $\sigma > 0$ be sufficiently
  small. By the same type of argument as in the construction of
  $v_\delta^-$ above, there exists a local minimizer $v_\delta^+$ of
  $E$, such that $v_\delta^+(y) = \sigma$ for all $y
  \in \partial\Omega_\pm$, and we have $v + \beta \leq v_\delta^+ \leq
  v + \tfrac14 \delta$, for some $\beta > 0$.

  Now, let $c \in (c^\dag_v, c^\dag)$, and consider
  $\Psi^{v_\delta^+}_c$ defined in \eqref{Psic} with $v_\delta^+$ in
  place of $v$.  Then, by an extension of the argument of Lemma
  \ref{lemN} it is not difficult to see that there exists a minimizer
  $\bar h$ of $\Psi^{v_\delta^+}_c$ in the set
  \[
  \begin{aligned}
    X := \Big\{ h\in H^1_c(\Sigma): &\ 0\le h\le 1-v_\delta^+,\,
    h=1-v_\delta^+\textup{ in }\Omega\times (-\infty,0],
    \\
    & h(y,z)= (1-v_\delta^+(y)) \eta(z) \textup{ for } (y, z)
    \in \partial\Omega_\pm\times \mathbb R \Big\},
  \end{aligned}
  \]
  where $\eta \in C^\infty(\mathbb R)$ is a cutoff function with the
  property that $\eta(z) = 1$ for all $z < 0$ and $\eta(z) = 0$ for
  all $z > 1$.  Indeed, semicontinuity and coercivity of
  $\Psi_c^{v_\delta^+}$ only depend on the behavior of the functional
  for large values of $z$. Since $v_\delta^+$ is still a local
  minimizer of $E$, the functional $\Psi_c^{v_\delta^+}$ is lower
  semicontinuous by \cite[Proposition 5.5]{lmn:arma08}. Furthermore,
  by hypothesis (H2) and Taylor formula
  \begin{eqnarray*}
    \Psi_c^{v_\delta^+}[h] = \Psi_c^v[h] + \int_\Sigma \int_0^h
    e^{c z} \Big( f_u(v  + s, y) - f_u(v_\delta^+ + s, y)
    \Big) (h - s) ds \, dx 
    \\ 
    \geq \Psi_c^v[h] - C ||v_\delta^+ - v||_{L^\infty(\Sigma)}^\gamma
    \int_\Sigma \int_0^h e^{c z}  (h - s) ds \, dx \\ 
    \geq \Psi_c^v[h] - C \delta^\gamma ||h||_{L^2_c(\Sigma)}^2,
  \end{eqnarray*}
  for some $C > 0$, implying coercivity of $\Psi_c^{v_\delta^+}$ for
  small enough $\delta$ by the argument of \cite[Proposition
  6.9]{lmn:arma08}. So the minimizer $\bar h$ of $\Psi_c^{v_\delta^+}$
  exists, and $\bar h(\cdot,z)\to 0$ uniformly in $\Omega$, as $z\to
  +\infty$ (indeed, the convergence is exponential by
  \cite[Proposition 3.3(iii)]{lmn:arma08}). Therefore, there exists $a
  > 0$ such that $\bar h(\cdot, z) \leq \tfrac14 \delta$ for all $z
  \geq a$.
   
  We finally let $u^+(y,z,t):=v_\delta^+(y) + \bar h(y, z - z_0 + b
  t)$, with $b:= c^\dag-c>0$, which is a supersolution for
  \eqref{pdec} on $\Sigma \times [0, T]$.  Notice that
  \begin{equation}\label{vip}
    u^+(\cdot, 0) =1 \qquad {\rm on\ } \Omega\times (-\infty,z_0],
  \end{equation} 
  and
  \begin{eqnarray}
    \label{eq:19}
    u^+(\cdot, z, t) \leq v + \frac{\delta}{2}  \qquad {\rm on\
    } \Omega \qquad \forall t \geq 0 ~~\forall z \geq z_0 +
    a - bt. 
  \end{eqnarray}

  \noindent{\it Comparison.} From \eqref{uzv} and \eqref{eq:vdel} for
  $\alpha$ small enough we have
  \[
  u^-(\cdot, 0) \le u_0 \qquad {\rm on\ }\Omega\times (-\infty,\bar
  z_0]\, .
  \]
  Also, by \eqref{zv}, \eqref{z0p} and \eqref{eq:vdel} for $\eta$
  small enough we have
  \begin{eqnarray}
    \label{z0v}
    u^-(\cdot, \bar z_0, t) \leq u(\cdot, \bar z_0, t)
    \qquad {\rm on\ 
    }\Omega \qquad \forall t\in [0,T]. 
  \end{eqnarray}
  Therefore, by parabolic comparison principle \cite{protter} we
  obtain
  \begin{equation}\label{eqqu} 
    u^- \le u \qquad {\rm on\ }\Omega\times
    (-\infty,\bar z_0]\times [0,T]\,. 
  \end{equation}
  In particular, by \eqref{eq:vdel} and the fact that by Theorem
  \ref{t:min} we have $\bar u(\cdot, z) < v$ for every $z \in \mathbb
  R$, it follows that
  \begin{eqnarray}
    \label{lbd}
   u(\cdot, z, t) \geq \bar u(\cdot, z - R(t)) - \delta \qquad
    \forall z \leq \bar z_0, ~\forall t \geq 0.    
  \end{eqnarray}
 
  On the other hand, in view of \eqref{vip}, the fact that $u^+ \geq v
  + \beta$, \eqref{linfcr} with $\eta$ replaced by $\eps$ at $t = 0$
  due to \eqref{eqe}, and the fact that \alert{$|R(0)|\le\delta$,} for
  every $z_0$ it is possible to choose $\eps$ small enough, so that
  \begin{equation}
    \label{eq:17}
    u_0 \leq u^+(\cdot, 0) \qquad {\rm on\ } \Sigma. 
  \end{equation}
  Then, by parabolic comparison principle we have
  \begin{equation}
    \label{eq:10}
    u \leq u^+  \qquad {\rm on\ } \Sigma \times [0,+\infty),
  \end{equation}
  and, possibly reducing $\bar z_0$ to ensure that $\bar u(\cdot, \bar
  z_0 + a + \delta) \geq v - \tfrac12 \delta$, in view of monotonicity
  of $\bar u(\cdot, z)$ by Theorem \ref{t:min}, for every $z_0 \leq
  \bar z_0$ we obtain
  \begin{eqnarray}
    \label{ubd}
    u(\cdot, z, t) \leq \bar u(\cdot, z - R(t)) + \delta \qquad
    \forall z \in [z_0 + a - bt, \bar z_0 + a] ~~ \forall t \geq
    0. 
  \end{eqnarray}
  Finally, combining \eqref{lbd} with \eqref{ubd} and \eqref{linf0},
  we get \eqref{w0infty}.
\end{proof}

We now prove a technical lemma that will be useful in the proof of
Proposition \ref{proh}.

\begin{Lemma}\label{lem:delta}
  There exist $0 < C_1 < C_2$, such that for all $|R| \leq 1$ we have
  \begin{equation}\label{eq:1}
    C_1 |R| \leq ||\TT_R \bar u - \bar u ||_{L^2_{\cd}(\Sigma)}
    \leq C_2 |R|.
  \end{equation}
  Furthermore, 
  \begin{eqnarray}
    \label{eq:12}
    ||\TT_R \bar u - \bar u ||_{L^2_{\cd}(\Sigma)} \geq C_1,
  \end{eqnarray}
  for all $|R| \geq 1$.
\end{Lemma}

\begin{proof}
  Let us first prove the upper bound. Notice that, thanks to Theorem
  \ref{t:min}, the functions $\bar u$ belongs to $H^1_{\cd}(\Sigma)$,
  hence in particular the map $\eta\mapsto \TT_\eta \bar u$ defines a
  differentiable curve in $L^2_{\cd}(\Sigma)$. A direct computation
  then gives
  \begin{eqnarray*}
    ||\TT_R \bar u - \bar u ||_{L^2_{\cd}(\Sigma)}
    &=&
    \left\|\int_{0}^{R} 
      \TT_\eta \bar u_z \,d\eta\right\|_{L^2_{\cd}(\Sigma)}
   \le 
    \int_{-|R|}^{|R|}
    ||\TT_\eta \bar u_z||_{L^2_{\cd}(\Sigma)}\,d\eta
    \\
    &\le&
    \left(\int_{-|R|}^{|R|}e^{\frac{\cd}{2}\eta}\,d\eta\right)
    ||\bar u_z||_{L^2_{\cd}(\Sigma)}
    \le C |R|,
  \end{eqnarray*}
  where we used the identity $||\TT_\eta \bar u_z||_{L^2_\cd(\Sigma)}
  = e^{\cd \eta \over 2} ||\bar u_z||_{L^2_\cd(\Sigma)}$.

  To obtain the lower bound in \eqref{eq:1}, we observe that for any
  $\Sigma_0 \Subset \Sigma$ compact, we have
  \begin{eqnarray}
    \label{eq:13}
    ||\TT_R \bar u - \bar u||_{L^2_{\cd}(\Sigma)}^2
    \geq \int_{\Sigma_0} e^{\cd z} (\bar u(y, z - R) - \bar u(y, z)
    )^2 dx  \nonumber \\ 
    =  R^2 \int_{\Sigma_0} e^{\cd z} \bar u_z^2 (y, z - \tilde R(y
    , z))
    \, dx, 
  \end{eqnarray}
  for some $0 < |\tilde R(y , z)| < |R|$. The lower bound then follows
  from the fact that $\bar u_z < 0$ in $\Sigma$ and, hence, $|u_z (y,
  z - \tilde R(y , z))|$ is bounded away from zero in $\Sigma_0$, as
  long as $|R| \leq 1$.  Finally, to get \eqref{eq:12} we observe that
  $||\TT_R \bar u - \bar u||_{L^2_{\cd}(\Sigma)}^2$ is monotonically
  increasing in $|R|$.
\end{proof}

We now look for a suitable translation of $\bar u$ which serves
as the best approximation, in some sense, to the solution of
\eqref{pdec}. For a given $u\in H^1_\cd(\Sigma)$ and $R\in\R$, we
define the function $h$ as:
\begin{eqnarray}
  \label{h}
  h(u,R):=\frac12 \int_\Sigma e^{\cd z}(u(y,z)-\bar u(y, z-R))^2 dx
  \geq 0. 
\end{eqnarray}
In the following proposition, we show that the optimal approximation
to $u$ can be naturally introduced by minimizing $h$ in \eqref{h} with
respect to $R$.

\begin{Proposition}\label{proh}
  For any $\delta >0$ sufficiently small there exists $\eps > 0$ such
  that, for any $u \in H^1_\cd(\Sigma)$ satisfying $\| u-\bar
  u\|_{L^2_\cd(\Sigma)} \leq \eps$ the function $h(u,\cdot)$ attains
  its global minimum. Furthermore, this minimum is unique, is
  contained in $(-\delta, \delta)$ and there are no other critical
  points in $(-\delta,\delta)$.
\end{Proposition}

\begin{proof}
  First of all, observe that $\eta\mapsto \bar u(y,z-\eta)$ is a twice
  differentiable curve in $L^2_\cd$, thanks to Theorem \ref{t:min}. By
  assumption
  \begin{eqnarray}
    \label{eq:14}
    \inf_{R \in \mathbb R} h(u, R) \leq h(u, 0) \leq \eps^2. 
  \end{eqnarray}
  Furthermore, from the lower bound in Lemma \ref{lem:delta} we get
  that
  \begin{eqnarray}
    \label{eq:15}
    C \min \{1, |R| \} \leq  ||\TT_R \bar u - \bar u
    ||_{L^2_{\cd}(\Sigma)} \qquad \qquad \nonumber \\ 
    \leq  ||u - \bar u||_{L^2_{\cd}(\Sigma)}
    + \sqrt{2 h(u, R)} \leq \eps + \sqrt{2 h(u, R)},
  \end{eqnarray}
  for some $C > 0$. Therefore, by continuity of $h(u, \cdot )$ its
  minimum is attained and lies in $(-\delta, \delta)$ for any $\delta
  > 0$, provided that $\eps$ is sufficiently small.  

  We now calculate the first and the second derivative of $h(u,\cdot)$
  with respect to $R$:
  \begin{eqnarray}
    \label{hp}
    h'(u, R) & = & \int_\Sigma e^{\cd z} (u(y,z)-\bar u(y,
    z- R)) \bar u_z(y, z - R) \, dx, \\
    \label{hpp}
    h''(u, R) & = & \cd h'(u, R) + \int_\Sigma e^{\cd z} u_z(y, z)
    \bar u_z(y, z - R)  \, dx,
  \end{eqnarray}
  \alert{where from now on the prime denotes the derivative with
    respect to $R$.}  Recalling Lemma \ref{lem:delta}, for all
  $|R|\leq \delta \leq 1$ we have
  \begin{eqnarray}
    |h'(u, R)| &\leq& C \, || u - \TT_R \bar u||_{L^2_\cd(\Sigma)} 
    \nonumber 
    \\
    &\leq& 
    C \left( ||u - \bar u||_{L^2_\cd(\Sigma)}
    \right. + \left.  || \bar u - \TT_R \bar
      u||_{L^2_\cd(\Sigma)}\right) 
    \nonumber
    \\ \label{hpest}
    &\le& C \, (\eps + |R|), 
  \end{eqnarray}
  for some $C > 0$. Now, observe that upon integration by parts we
  have
  \begin{equation*}
    \begin{aligned}
      & \int_\Sigma e^{\cd z} \bar u_z(y, z - R) (u_z(y, z) - \bar
      u_z(y, z - R)) \, dx \\ 
      & = - \int_\Sigma e^{\cd z} (\bar
      u_{zz}(y, z - R) + \cd \bar u_z(y, z - R)) (u(y, z) - \bar u(y,
      z - R)) \, dx .
    \end{aligned}
  \end{equation*}
  Therefore, since $\bar u_z \in H^1_\cd(\Sigma)$ by Theorem
  \ref{t:min}, applying Cauchy-Schwarz inequality we obtain
  \begin{eqnarray}
    \label{uzuz}
    \left| \int_\Sigma e^{\cd z} \bar u_z(y, z - R)  (u_z(y, z) -
      \bar u_z(y, z - R)) \,
      dx \right| \nonumber \\ \leq C e^{\cd R/2} ||u - \TT_R \bar
    u||_{L^2_\cd(\Sigma)},  
  \end{eqnarray}
  for some $C > 0$. Applying this estimate to \eqref{hpp} and
  combining it with the estimates in \eqref{hpest}, we obtain
  \begin{eqnarray}\label{ellade}
    h''(u, R)  &\ge&  
    \alert{e^{\cd R}} \,|| \bar u_z||^2_{L^2_\cd(\Sigma)} - C(\eps +
    |R|),  
  \end{eqnarray}
  for some constant $C>0$. This implies that $h''(u, R) \geq M$ for
  some $M > 0$ and all $|R| \leq \delta$, provided that $\delta$ and
  $\eps$ are small enough.  Hence $h(u, \cdot)$ is a strictly convex
  function on $[-\delta, \delta]$, and the minimum of $h(u, \cdot)$ is
  the unique critical point in $(-\delta, \delta)$.
\end{proof}

~

\alert{Recalling the comment following Lemma \ref{lemlin}, we can now
  formulate a nonlinear analog of the result of that lemma.

\begin{Remark}\rm 
  Suppose that $u$ is sufficiently close to $\bar u$ in
  $L^2_\cd(\Sigma) \cap L^\infty(\Sigma)$. Then by Proposition
  \ref{proh} there exists $R_0 \in \mathbb R$, such that the function
  $h(u, R)$ in \eqref{h} is minimized with respect to $R$ at $R =
  R_0$. Therefore, we have that $u - \TT_{R_0} \bar u$ is orthogonal
  to $\TT_{R_0} \bar u$ in $L^2_\cd(\Sigma)$, and so by Lemma
  \ref{lemlin}, hypothesis (H2) and the minimizing property of $\bar
  u$ we have $\Phi_\cd[u] \geq \tfrac12 K ||u - \TT_{R_0} \bar
  u||_{L^2_\cd(\Sigma)}^2$, where $K > 0$ is as in Lemma
  \ref{lemlin}. Hence $\bar u$ is, in fact, a strict local minimizer
  of $\Phi_\cd$ in the above sense.
\end{Remark}}

We now conclude the proof of Theorem \ref{t:wnstab}.

\medskip
\emph{Proof of Theorem \ref{t:wnstab}.}
  Let $\delta > 0$ be sufficiently small, so that Proposition
  \ref{proh} applies with $u = u_0$ and all $0 < \eps \leq \eps_0$,
  for some $\eps_0 > 0$. Then by Propositions \ref{p:ini} and
  \ref{proh} there exists $T_0 > 0$, such that there exists a
  minimizer $R(t)$ of $h(u(\cdot, t), R)$ in $R$ for each $t \in [0,
  T_0]$. Furthermore, $R(t)$ is the unique critical point of
  $h(u(\cdot, t), R)$ in $(-\delta, \delta)$. In fact, $R(t)$ is a
  continuously differentiable function of $t$ on $[0, T_0]$. Indeed,
  since $R(t)$ minimizes $h(u(\cdot, t), R)$ in $R$, we have 
  \begin{eqnarray}
    \label{hp0}
    h'(u(\cdot, t), R) = \int_\Sigma e^{\cd z} (u(y,z, t)-\bar u(y,
    z - R)) \bar u_z(y, z - R) \, dx = 0,
  \end{eqnarray}
  whenever $R = R(t)$. In view of the continuity of $u_t(\cdot, t)$ in
  $L^2_\cd(\Sigma)$ guaranteed by Proposition \ref{p:ini}, as well as
  Theorem \ref{t:min} and Lemma \ref{lem:delta}, the function in
  \eqref{hp0} is continuously differentiable in $R$ and $t$ in some
  small neighborhood of the origin. Then, arguing as in Proposition
  \ref{proh} one can see that $h''(u(\cdot, t), R) > 0$ there, so we
  can apply the implicit function theorem to \eqref{hp0}.
  Furthermore, after some algebra we obtain
  \begin{eqnarray}
    \label{Rt}
    {dR(t) \over dt} = -{\int_\Sigma e^{\cd z} u_t(y, z, t) \bar
      u_z(y, z - R(t)) \, 
      dx \over \int_\Sigma e^{\cd z} u_z(y, z, t) \bar u_z(y, z - 
      R(t)) \, dx}.
  \end{eqnarray}

  For \alert{$t \in (0, T_0]$ and $u$ solving (\ref{pdec})} we define
  \begin{eqnarray}
    w(y, z, t) := u(y, z, t) - \bar u(y, z - R(t)).
  \end{eqnarray}
  The function $w$ satisfies the equation
  \begin{eqnarray}
    \label{w}
    w_t= \Delta w +\cd w_z + {dR \over dt} \, \bar u_z + f_u(\tilde u,
    y) w, 
  \end{eqnarray}
  for some $\tilde u$, with $|\tilde u - \TT_{R(t)} \bar u|\le
  |w|$. Also, by construction we have
  \begin{eqnarray}
    \label{ort}
    \int_\Sigma e^{\cd z} w(y, z, t) \bar u_z(y, z - R(t)) \, dx = 0.
  \end{eqnarray} 
  We now introduce
  \[
  m(t):= \int_\Sigma e^{\cd z} w^2(x, t) \, dx \qquad t\ge 0,
  \]
  so that $m(0) \leq \eps^2$.  Multiplying (\ref{w}) with $e^{\cd z}
  w$ and integrating over $\Sigma$, we obtain
  \begin{eqnarray}
    \label{norm2}
    \frac{dm(t)}{dt}  = -2\int_\Sigma
    e^{\cd z}  \left( |\nabla w|^2 - f_u(\tilde u, y) w^2  \right) dx
  \end{eqnarray}
  where we used (\ref{ort}) to erase the term multiplying $dR/dt$.  By
  Lemma \ref{lemlin} we have
  \begin{equation}\label{normdue}
    \begin{aligned}
      \int_\Sigma e^{\cd z} \left( |\nabla w|^2 - f_u(\tilde u, y) w^2
      \right) dx =& \int_\Sigma e^{\cd z} \left( |\nabla w|^2 -
        f_u(\bar u, y) w^2 \right) dx
      \\
      &+ \int_{\Sigma} e^{\cd z} \left( f_u(\bar u,y) - f_u(\tilde u,
        y) \right) w^2 dx
      \\
      \ge& K m(t) + \int_{\Sigma} e^{\cd z} \left( f_u(\bar u,y) -
        f_u(\tilde u, y) \right) w^2 dx .
    \end{aligned}
  \end{equation}
  Possibly reducing $\delta$ and recalling assumption (H2), we have
  \[
  \int_{\{ |w|<\delta\}} e^{\cd z} \left| f_u(\bar u,y) - f_u(\tilde
    u, y) \right| w^2 dx \le \frac K2 m(t),
  \]
  which, combined with \eqref{normdue}, gives
  \[
  \begin{aligned}
    \int_\Sigma e^{\cd z} \left( |\nabla w|^2 - f_u(\tilde u, y) w^2
    \right) dx \ge & \frac K2 m(t) - \int_{\{|w|\ge\delta\}} e^{\cd z}
    \left| f_u(\bar u,y) - f_u(\tilde u, y) \right| w^2 dx
    \\
    \ge & \frac K2 m(t) - \int_{-\infty}^{z_\delta(t)}\int_{\Omega}
    e^{\cd z} \left| f_u(\bar u,y) - f_u(\tilde u, y) \right| w^2 \,
    dy \, dz
    \\
    \ge & \frac K2 m(t) - C e^{\cd z_\delta(t)},
  \end{aligned}
  \]
  for some $C>0$, where $z_\delta(t)$ is defined in \eqref{zeta}.

  We can now apply Proposition \ref{p:w2close} with
  $z_0=\frac{4}{\cd}\log\omega$, which yields some $\eps > 0$ and
  $\eta > 0$, with $\eta$ independent of $\omega$, and $T \in (0,
  T_0]$ depending on $\eta$.  We then get
  \begin{equation}\label{norme}
    \begin{aligned}
      \int_\Sigma e^{\cd z} \left( |\nabla w|^2 - f_u(\tilde u, y) w^2
      \right) dx \ge& \frac K2 m(t) - C e^{\cd(z_0-bt)} \ge \frac K2
      m(t) - C \omega^4 e^{-\cd bt}
    \end{aligned}
  \end{equation}
  for some $b > 0$ and $C > 0$ and all $t\in [0,T]$.  From
  \eqref{norm2} and \eqref{norme} we, therefore, obtain
  \begin{equation}\label{mdt}
    {dm(t) \over dt} \le - K m(t) + 2C \omega^4 e^{-\cd bt}
    \qquad \forall t \in [0, T],
  \end{equation}
  which gives
 \begin{eqnarray}
    \label{exp}
    ||w(\cdot,t)||_{L^2_\cd(\Sigma)}
    \le 
    M \omega^2 e^{-\sigma t}
    \qquad \forall t \in [0,T],
  \end{eqnarray}
  for some $\sigma > 0$ and $M > 0$, provided that $\eps$ is small
  enough. 
  
  To estimate the behavior of $R(t)$, \alert{we substitute $u_t=w_t-
    \bar u_z dR/dt$ into (\ref{Rt}) and take into account (\ref{w})
    and (\ref{trav}) differentiated in $z$, noting that $\bar u_z \in
    H^2_c(\Sigma)$ by Theorem \ref{t:min}. After a few integrations by
    parts we obtain}
  \begin{align}
    \label{Rtt}
    {dR(t) \over dt}  = {\int_\Sigma e^{\cd z} (f_u(\bar
      u(y, z - R(t)),y) - f_u(\tilde u, y)) w(y, z, t) \bar u_z(y, z -
      R(t)) dx  \over
      \int_\Sigma e^{\cd z} (\bar u_z(y, z - R(t)) + w_z(y, z, t)) \bar
      u_z(y, z - R(t)) dx}.  
  \end{align}
  By the same argument as the one leading to \eqref{uzuz}, we have
  \begin{eqnarray}
    \label{uzwz}
    \left| \int_\Sigma e^{\cd z} \bar u_z(y, z - R(t)) w_z(y, z, t) \,
      dx \right| \leq C e^{\cd R(t)/2} ||w(\cdot, t)||_{L^2_\cd(\Sigma)}.
  \end{eqnarray} 
  With hypothesis (H2), this leads to the following estimate for
  $dR/dt$:
  \begin{eqnarray}
    \label{Rte}
    \left| {dR(t) \over dt} \right| &\le& {C ||w(\cdot,
      t)||_{L^2_\cd(\Sigma)} \over \widetilde C e^{\cd R(t)/2}
      -  ||w(\cdot, t)||_{L^2_\cd(\Sigma)}}
  \end{eqnarray}
  for some constants $C,\widetilde C > 0$, \alert{provided that
    $\omega$ is so small that by (\ref{exp}) the denominator in
    (\ref{Rte}) is positive for all $t\in [0,T]$.} Then we have
  \begin{equation}\label{Rdt}
    \left| \frac{d R(t)}{d t} \right| \leq C \omega^2 e^{-\sigma t}
    \qquad \forall t \in [0, T],
  \end{equation}
  and hence
  \begin{equation}
    \label{modRt}
    |R(t)| \leq \widetilde M  \omega^2 \leq\omega
    \qquad \forall t \in [0, T],
  \end{equation}
  for some $\widetilde M > 0$ and $\omega$ small enough.  Moreover, since
  $\eta$ and $\delta$ are independent of $\omega$, \eqref{exp} also implies that
  \eqref{equesto} holds uniformly in $T$ for $\omega$ small enough, whence  
  $T = T_0$. Indeed, if $T_1 < T_0$ is the maximum
  value of $T$ for which Proposition \ref{proh} can be applied, then
  by \eqref{exp} the left-hand side of \eqref{equesto} is bounded by
  $\eta/2$ at $t = T_1$, provided that $\omega$ is sufficiently
  small. Therefore, by continuity of $w(\cdot, t)$ in
  $L^2_\cd(\Sigma)$ guaranteed by Proposition \ref{p:ini}, the
  inequality in \eqref{equesto} also holds for some interval beyond
  $T_1$, contradicting the maximality of $T_1$.

  Moreover, by \eqref{modRt} the function $R(t)$ is in fact defined
  and continuously differentiable for all $t \geq 0$. Indeed, let us
  take $T_0$ to be the largest possible value for which $||u(\cdot, t)
  - \bar u||_{L^2_\cd(\Sigma)} \leq \eps_0$ for all $t \in [0, T_0]$,
  so that Proposition \ref{proh} still applies. In view of Lemma
  \ref{lem:delta}, \eqref{exp} and \eqref{modRt}, we have
  \begin{eqnarray}
    \label{eq:16}
    ||u(\cdot, t) - \bar u||_{L^2_\cd(\Sigma)} &\leq& ||u(\cdot, t) -
    \TT_{R(t)} \bar u||_{L^2_\cd(\Sigma)} \nonumber 
    \\ 
    &&\ + ||\bar u -
    \TT_{R(t)} \bar u||_{L^2_\cd(\Sigma)} 
    \\ \nonumber
    &\leq& M \omega^2 \qquad \forall t \in [0, T_0],
  \end{eqnarray}
  for some $M > 0$.  Therefore, choosing $\omega$ so small that the
  right-hand side of \eqref{eq:16} is bounded by $\tfrac12 \eps_0$
  and, once again, taking into account continuity of $w(\cdot, t)$ in
  $L^2_\cd(\Sigma)$, we can then make sure that the assumptions of
  Proposition \ref{proh} are satisfied on some interval beyond $T_0$,
  contradicting maximality of $T_0$. We thus proved that we can take
  an arbitrarily large $T_0 > 0$ in all the arguments above.

  Finally, using \eqref{exp} and \eqref{Rte} again and keeping in mind
  that by \eqref{modRt} the denominator in (\ref{Rte}) is bounded away
  from zero, we finally obtain that the limit $\displaystyle R_\infty
  := \lim_{t \to +\infty} R(t)$ exists, and recalling Lemma
  \ref{lem:delta} we have
  \begin{eqnarray}\label{est}
    ||u(\cdot, t) - \TT_{R_\infty} \bar u||_{L^2_\cd(\Sigma)} \leq 
    \omega e^{-\sigma t} \qquad \forall t \geq 0,
  \end{eqnarray}
  for some $\sigma > 0$, provided that $\omega$ is small enough,
  yielding the thesis of the theorem.
  \endproof

\section{Proof of the main result}\label{sec:main}

We will prove Theorem \ref{t:main} in the reference frame moving with
speed $\cd$, that is, we will prove that if $u$ is the solution of
\eqref{pdec} with the initial datum satisfying the assumptions of
Theorem \ref{t:main}, then it converges in $H^2_\cd(\Sigma)$ to
$\TT_{R_\infty} \bar u$ for some $R_\infty$ as $t \to \infty$. The
result then follows by noting that $\TT_{\cd t} u$ solves \eqref{pde}
with the same initial condition, upon applying $\TT_{-R_\infty}$.

  From now on, $u$ always refers to the solution of \eqref{pdec}. We
  divide the proof into five steps.

  \paragraph{Step 1} We begin by constructing an appropriate pair of
  barrier solutions of (\ref{pdec}) to ensure that the solution of the
  initial value problem for (\ref{pdec}) does not move too far towards
  the ends of the cylinder. The barriers are obtained by considering
  the solutions $\bar u^\pm$ of (\ref{pdec}) with the initial data
  \begin{eqnarray}
    \label{eq:u_p}
    \bar u_0^-(y, z) & = & \min \{ u_0(y, z - R), \bar u(y, z) \}, \\ 
    \bar u_0^+(y, z) & = & \max \{ u_0(y, z + R), \bar u(y, z) \},
  \end{eqnarray}
  where $R > 0$ is so big that both $\bar u_0^\pm$ satisfy the
  assumptions of Theorem \ref{t:wnstab}, provided that $\alpha$ in
  \eqref{inidat} is small enough. Indeed, by definition and
  \eqref{eqbar} the assumption in \eqref{inidat} is satisfied for both
  $\bar u_0^\pm$. Moreover, for $\bar u_0^+ - \bar u = \max (\TT_{-R}
  u_0 - \bar u, 0)$ we have
  \begin{eqnarray}
    0 \leq \bar u_0^+ - \bar u  \leq \TT_{-R} u_0 \to 0
    ~~\mathrm{in}~ L^2_\cd(\Sigma) ~~\mathrm{as}~R \to +\infty,
  \end{eqnarray}
  so $\bar u_0^+ \to \bar u$ in $L^2_\cd(\Sigma)$ as $R \to
  +\infty$. By a similar argument for $\bar u - \bar u_0^- = \max
  (\bar u - \TT_R u_0, 0)$ we have
  \begin{eqnarray}
    0 \leq \bar u - \bar u_0^- \leq \bar u \to 0 
    ~~\mathrm{in}~ L^2_\cd(\Omega \times (M, +\infty))
    ~~\mathrm{as}~M \to +\infty,
  \end{eqnarray}
  uniformly in $R$.  At the same time, by boundedness of $\bar u_0^-$
  and $\bar u$ we have $||\TT_R u_0 - \bar u||_{L^2_\cd(\Omega \times
    (-\infty, -M))} \to 0$ as $M \to +\infty$, again, uniformly in
  $R$.  Finally, in view of \eqref{inidat}, \eqref{eqbar} and the Hopf
  lemma, for every $M > 0$ and $R > 0$ large enough we have $|\{ x \in
  \Omega \times (-M, M): \bar u_0^-(x) < \bar u(x) \}| \leq C \alpha$,
  with some $C = C(M) > 0$, for small enough $\alpha$.  Therefore, it
  is possible to choose $M$ large enough, then $R$ large enough, and
  then $\alpha$ small enough, so that $||\bar u - \bar u_0^-
  ||_{L^2_\cd(\Sigma)}$ can be made as small as desired. Note that
  both functions $\bar u_0^\pm$ obtained above satisfy \eqref{inidat}
  uniformly in $R$.

  We now claim that $\bar u^\pm(y, z \mp R, t)$, i.e., the solutions
  of \eqref{pdec} with initial data $\bar u_0^\pm(y, z \mp R)$, are
  the appropriate barrier solutions. Indeed, by construction the
  initial data $u_0$ is sandwiched between $\bar u^\pm(y, z \mp R, t)$
  at $t = 0$, hence by parabolic comparison principle \cite{protter}
  the solution of (\ref{pdec}) will remain so for all times. By
  Theorem \ref{t:wnstab} we know that there exist $R_\infty^\pm$ such
  that
  \begin{equation}\label{eqqa}
    || \bar
    u^\pm(y,z\mp R,t)-\bar u(y, z \mp
    R_\infty^\pm)||_{L^2_\cd(\Sigma)} \le e^{-\sigma t}, 
  \end{equation}
  for some $\sigma>0$ and any $z_0 \in \mathbb R$, provided that
  $\alpha$ is sufficiently small and $R$ is sufficiently large.

  \paragraph{Step 2} We now use the functional $\Phi_\cd$ as a
  Lyapunov functional to establish existence of a sequence $t_n \to
  +\infty$ on which $u(\cdot, t_n)$ converges to a translate of $\bar
  u$.  Indeed, multiplying (\ref{pdec}) by a test function
    $\varphi \in C^\infty_0(\mathbb R^n)$ vanishing on $\partial
    \Sigma_{\pm}$ and integrating over $\Sigma$, we can write
    (\ref{pdec}) in the weak form as
  \begin{eqnarray}
    \label{pdecw}
    \int_\Sigma e^{\cd z} \varphi \, u_t \,dx = - \int_\Sigma
    e^{\cd z} \left( \nabla  u \cdot \nabla \varphi - f(u, y) \varphi
    \right) dx,   
  \end{eqnarray}
  where the integral in the right-hand side is the G\^ateaux
  derivative of $\Phi_\cd$ at $u(\cdot, t)$ in the direction of
  $\varphi$. Therefore, (\ref{pdec}) is the gradient flow generated by
  $\Phi_\cd$ in $L^2_\cd(\Sigma)$, and in view of \eqref{h3c} for all
  $t_2\ge t_1 > 0$, we have
  \begin{equation}\label{eqenergy}
    \Phi_\cd[u(\cdot,t_1)] - \Phi_\cd[u(\cdot,t_2)] =  
    \int_{t_1}^{t_2}\| u_{t}(\cdot,t)\|^2_{L^2_\cd(\Sigma)}dt.
  \end{equation}
  Letting $t_2\to +\infty$ and recalling that $\Phi_\cd[u]$ is bounded
  below by Theorem \ref{t:min}, from \eqref{eqenergy} it follows that
  there exists a sequence $t_n\to+\infty$ such that
  \begin{equation}\label{equt}
    \lim_{n\to +\infty}\| u_{t}(\cdot,t_n)\|_{L^2_\cd(\Sigma)} = 0.
  \end{equation}
  Also note that, since $0\le u(y,z,t)\le \bar u^+(y, z - R,t)$ for all
  $(y,z)\in\Sigma$ and $t\ge 0$, and $\bar u^+(\cdot, t)$ is uniformly
  bounded in $L^2_\cd(\Sigma)$ by Theorem \ref{t:wnstab}, in view of
  hypotheses (H1)--(H2) we have that
  \begin{equation}\label{h1bound}
    \| u\|_{H^1_\cd(\Sigma)}^2 \leq 2 \Phi_\cd[ u] +
    C \| u\|_{L^2_\cd(\Sigma)}^2 \leq M(t_0) \qquad \forall t\ge
    t_0>0 
  \end{equation}
  for some constants $C, M(t_0) > 0$.  

  From \eqref{equt} and \eqref{h1bound}, up to a possible subsequence,
  we can pass to the limit in \eqref{pdecw} and get that $ u(\cdot,
  t_n)$ converges to a critical point of $\Phi_\cd$ weakly in
  $H^1_\cd(\Sigma)$. In fact, the limit must be a non-trivial critical
  point $u_\infty$ of $\Phi_\cd$, in view of Step 1, hence a translate
  of $\bar u$ by \cite[Propositions 3.2 and 3.5]{lmn:arma08} and
  Theorem \ref{t:min}.

  \paragraph{Step 3} We now prove that $u(\cdot, t_n)\to u_\infty$ in
  $L^2_\cd(\Sigma)$.  Notice first that since both $u$ and $u_\infty$
  are uniformly bounded, for a given $\eps>0$ we can find $M$ such
  that
  \begin{equation}\label{stimauno}
    \| u(\cdot, t_n)-u_\infty\|_{L^2_\cd(\Omega\times
      (-\infty,-M])}\le \eps . 
  \end{equation}
  Moreover, since $u(y,z, t)\le \bar u^+(y,z-R,t)$, from \eqref{eqqa}
  it also follows that
  \begin{eqnarray}\label{stimadue}
    \nonumber
    \| u(\cdot, t_n)\|_{L^2_\cd(\Omega\times (M,+\infty))} &\le&
    \| \bar u^+(\cdot, t_n)\|_{L^2_\cd(\Omega\times (M,+\infty))} 
    \\ &\le&
    \| \bar u(y,z-R^+_\infty)\|_{L^2_\cd(\Omega\times (M,+\infty))} 
    \\ \nonumber
    && +
    \| \bar u^+(\cdot, t_n)-\bar u(y,z-R^+_\infty)\|_{L^2_\cd(\Sigma)} 
    \\ \nonumber
    &\le& \eps,  
  \end{eqnarray}
  for $M$ big enough and all $n \geq N$, for some $N = N(M) \in
  \mathbb N$. Recalling that $H^1_\cd(\Sigma)$ compactly embeds into
  $L^2_\cd(\Omega\times (-M,M))$, from Proposition \ref{p:ini},
  \eqref{stimauno} and \eqref{stimadue} we obtain that
  $$
  u(\cdot,t_n)\to u_\infty \qquad {\rm in\ } L^2_\cd(\Sigma).
  $$ 

  \paragraph{Step 4}
  Take $n$ big enough so that
  \begin{eqnarray}
    \label{eq:21}
    \|u(\cdot,t_n)- u_\infty\|_{L^2_\cd(\Sigma)}\le \eps \qquad\forall
    t\ge t_n\,, 
  \end{eqnarray}
  where $\eps$ is the same as the one corresponding to $\omega = 1$ in
  Theorem \ref{t:wnstab}. On the other hand, for every $\alpha' > 0$
  it is possible to choose $\delta \leq \alpha'$ in Proposition
  \ref{p:w2close}, such that the subsolution $u^-$ constructed there
  satisfies $u^-(\cdot, z, t) \geq v - \alpha'$ for all $z \leq \bar
  z_0$, with some $\bar z_0 \in \mathbb R$ independent of $\eps$ and
  $R$ in the definition of $\bar u^-$, and all $t \geq 0$, if $\alpha$
  is sufficiently small. Therefore, we have $\bar u^- \geq u^-$ in
  $\Omega \times (-\infty, \bar z_0] \times [0, +\infty)$, and since
  $\bar u^- \leq u$ for all $t \geq 0$, the same inequality holds for
  $u$. So we can apply Theorem \ref{t:wnstab} to $u(\cdot, t_n)$ in
  place of $u_0$ (also applying suitable translations in $z$ and $t$),
  and obtain
  \begin{equation}\label{eleq}
    \|u(\cdot, t)- u_\infty\|_{L^2_\cd(\Sigma)}\le e^{-\sigma (t - t_n)} 
  \end{equation}
  for some $\sigma > 0$ independent of $u_0$ and all $t \geq t_n$.

  \paragraph{Step 5} 
  We now demonstrate that the exponential convergence of \eqref{eleq}
  also holds in spaces of higher regularity. We first show this for
  $H^1_\cd(\Sigma)$, and then for $H^2_\cd(\Sigma)$. In the following,
  we denote by $\mathcal A: \mathcal D(\mathcal A) \to
  L^2_\cd(\Sigma)$ the sectorial operator $\mathcal A = \Delta +
  \cd \partial_z$, with domain $\mathcal D(\mathcal A) =
  H^2_\cd(\Sigma)$ dense in $L^2_\cd(\Sigma)$ (see also
  \cite{mn:cms08,lunardi}).

  Letting $w(\cdot,t):= u(\cdot, t)-u_\infty$, we have
  \begin{equation}\label{elasso}
    w_t = \mathcal A w + g(x, t) w,
  \end{equation}
  where $g(y,z,t) = f_u(\tilde u(y,z,t),y)$ for some $\tilde u$ such
  that $|\tilde u-u_\infty|\le |w|$, i.e., $g$ is such that $|g|\le
  C$, for some $C>0$.  As a consequence, by parabolic regularity
  theory \cite[Chapter 15]{taylor3} (see also \cite[Proposition 2.1.1
  and Theorem 3.1.1]{lunardi}) and recalling \eqref{eleq}, for all $t
  \geq 1$ we have
  \begin{equation}\label{eqevans}
  \begin{aligned}
    \| u(\cdot, t)-u_\infty\|_{H^1_\cd(\Sigma)}& \le C \left( \|
      w(\cdot,t-1)\|_{L^2_\cd(\Sigma)} + \int_{t-1}^t \frac{\|
        w(\cdot,s)\|_{L^2_\cd(\Sigma)}}{\sqrt{t-s}}\,ds \right)
    \\
    &\le C e^{-\sigma t},
  \end{aligned}
  \end{equation}
  for some $C>0$. In particular, from \cite[Proposition
  3.2]{lmn:arma08}, \eqref{eqevans}, the minimizing property of
  $u_\infty$ and hypothesis (H2) we get
  \begin{eqnarray}\label{eqfici}
    \Phi_\cd[u(\cdot,t)] &=& \Phi_\cd[u(\cdot,t)]-\Phi_\cd[u_\infty] 
    \nonumber \\ 
    &=& 
    \int_\Sigma e^{\cd z}\left(\frac{1}{2} |\nabla w|^2
      +V(u_\infty + w,y) - V(u_\infty,y) - V'(u_\infty, y) w \right)
    dx  
    \nonumber \\ 
    &\le& C \|w\|_{H^1_\cd(\Sigma)}^2 \le Ce^{-2 \sigma t},	
  \end{eqnarray}
  for all $t \geq t_0$, with any $t_0 > 0$ and some $C = C(t_0) > 0$.
   
  Let us now rewrite \eqref{elasso} in the form
  \begin{equation}\label{elassobis}
    w_t = \mathcal A w + h(\cdot, t), \qquad   h(\cdot,t):=
    f(u(\cdot,t))-f(u_\infty). 
  \end{equation}
  Recalling hypothesis (H2), \eqref{eqenergy} and \eqref{eqfici}, for
  all $t_2\ge t_1\ge t_0 > 0$ we have
  \begin{eqnarray*}
    \|h(\cdot,t_2)-h(\cdot,t_1)\|_{L^2_\cd(\Sigma)} &=& 
    \| f(u(\cdot,t_2))-f(u(\cdot,t_1))  \|_{L^2_\cd(\Sigma)} 
    \\
    &\le& C \| u(\cdot,t_2)-u(\cdot,t_1)\|_{L^2_\cd(\Sigma)} 
    \\
    &\le& C \int_{t_1}^{t_2}  \| u_t(\cdot,s)\|_{L^2_\cd(\Sigma)} ds
    \\
    &\le& C\sqrt{(t_2-t_1) \Phi_\cd[u(\cdot, t_1)]} \le
    C\sqrt{t_2-t_1}\, 
    e^{-\sigma t_1},  
  \end{eqnarray*}
  for some $C = C(t_0) > 0$.  Then, reasoning as in \cite[Theorem
  4.3.1]{lunardi} with $t_1 = t - 1$ and $t_2 = t$ and using
  \eqref{eqevans}, we have
  \begin{eqnarray}
    \label{H2cest}
    \|w(\cdot,t)\|_{H^2_\cd(\Sigma)}
    &\le& C (\|\mathcal A w(\cdot,t)\|_{L^2_\cd(\Sigma)} + ||w(\cdot,
    t)||_{H^1_\cd(\Sigma)}  )
    \nonumber \\
    &\le & C  \bigg( \| w(\cdot,t - 1)\|_{L^2_\cd(\Sigma)}
    + \| w(\cdot,t)\|_{H^1_\cd(\Sigma)}
    \nonumber \\
    && \qquad +
    \int_{t-1}^{t}\frac{\|h(\cdot,s)-h(\cdot,t)\|_{L^2_\cd(\Sigma)}} 
    {t-s} \, ds \bigg)
    \nonumber \\
    &\le& C\,e^{-\sigma t},
  \end{eqnarray}
  for all $t \geq t_0 + 1$, for any $t_0 > 0$ and some $C = C(t_0) >
  0$. In writing \eqref{H2cest}, we used the same reasoning as in the
  standard estimate of the $H^2$-norm of a function in terms of the
  $L^2$-norm of the Laplacian to obtain the inequality in the first
  line. This gives \eqref{eqmain} and concludes the proof of Theorem
  \ref{t:main}.
  \endproof


\section*{Acknowledgements}
The authors wish to thank the University of Tours and the Research
Institute \textit{Le Studium} for the kind hospitality and support
(M.N.). We would also like to acknowledge valuable discussions with
V. Moroz.


\bibliographystyle{siam}
 
\bibliography{../mura,../nonlin}

\begin{thebibliography}{10}

\bibitem{barenblatt57}
{\sc G.~I. Barenblatt and Y.~B. Zeldovich}, {\em On the stability of flame
  propagation}, Prikladnaya Matematika i Mekhanika (Applied Mathematics and
  Mechanics (PMM)), 21 (1957), pp.~856--859.
\newblock (in Russian).

\bibitem{berestycki92arma}
{\sc H.~Berestycki, B.~Larrouturou, and J.-M. Roquejoffre}, {\em Stability of
  travelling fronts in a model for flame propagation. {I}. {L}inear analysis},
  Arch. Rational Mech. Anal., 117 (1992), pp.~97--117.

\bibitem{berestycki92}
{\sc H.~Berestycki and L.~Nirenberg}, {\em Traveling fronts in cylinders}, Ann.
  Inst. H. Poincar\'e Anal. Non Lin\'eaire, 9 (1992), pp.~497--572.

\bibitem{bramson}
{\sc M.~Bramson}, {\em Convergence of solutions of the Kolmogorov equation to
  travelling waves}, Amer. Math. Soc., Providence, RI, 1983.

\bibitem{evanspde}
{\sc L.~C. Evans}, {\em Partial Differential Equations}, vol.~19 of Graduate
  Studies in Mathematics, American Mathematical Society, Providence, RI, 1998.

\bibitem{fife79}
{\sc P.~C. Fife}, {\em Mathematical Aspects of Reacting and Diffusing Systems},
  Springer-Verlag, Berlin, 1979.

\bibitem{fife77}
{\sc P.~C. Fife and J.~B. McLeod}, {\em The approach of solutions of nonlinear
  diffusion equations to traveling front solutions}, Arch. Rat. Mech. Anal., 65
  (1977), pp.~335--361.

\bibitem{gallay09}
{\sc T.~Gallay and R.~Joly}, {\em Global stability of travelling fronts for a
  damped wave equation with bistable nonlinearity}, Ann. Sci. \'Ec. Norm.
  Sup\'er. (4), 42 (2009), pp.~103--140.

\bibitem{gallay07}
{\sc T.~Gallay and E.~Risler}, {\em A variational proof of global stability for
  bistable travelling waves}, Differential Integral Equations, 20 (2007),
  pp.~901--926.

\bibitem{gilbarg}
{\sc D.~Gilbarg and N.~S. Trudinger}, {\em Elliptic Partial Differential
  Equations of Second Order}, Springer-Verlag, Berlin, 1983.

\bibitem{heinze89}
{\sc S.~Heinze}, {\em Travelling waves for semilinear parabolic partial
  differential equations in cylindrical domains}, SFB 123 Preprint 506,
  University of Heidelberg, 1989.

\bibitem{heinze}
\leavevmode\vrule height 2pt depth -1.6pt width 23pt, {\em A variational
  approach to traveling waves.}, Tech. Rep.~85, Max Planck Institute for
  Mathematical Sciences, Leipzig, 2001.

\bibitem{henry05}
{\sc D.~Henry}, {\em Perturbation of the boundary in boundary-value problems of
  partial differential equations}, vol.~318 of London Mathematical Society
  Lecture Note Series, Cambridge University Press, Cambridge, 2005.

\bibitem{kanel60}
{\sc Y.~I. Kanel'}, {\em The behavior of solutions of the {C}auchy problem when
  the time tends to infinity, in the case of quasilinear equations arising in
  the theory of combustion}, Soviet Math. Dokl., 1 (1960), pp.~533--536.

\bibitem{kanel62}
\leavevmode\vrule height 2pt depth -1.6pt width 23pt, {\em On the stabilization
  of solutions of the {Cauchy} problem for the equations arising in the theory
  of combusion}, Mat. Sbornik, 59 (1962), pp.~245--288.

\bibitem{kirchgaessner92}
{\sc K.~Kirchg\"assner}, {\em On the nonlinear dynamics of travelling fronts},
  J. Differential Equations, 96 (1992), pp.~256--278.

\bibitem{kolmogorov37}
{\sc A.~N. Kolmogorov, I.~G. Petrovsky, and N.~S. Piskunov}, {\em Study of the
  diffusion equation with growth of the quantity of matter and its application
  to a biological problem}, Bull. Univ. Moscow, Ser. Int., Sec. A, 1 (1937),
  pp.~1--25.

\bibitem{lions84a}
{\sc P.-L. Lions}, {\em The concentration-compactness principle in the calculus
  of variations. {T}he locally compact case. {I}}, Ann. Inst. H. Poincar\'e
  Anal. Non Lin\'eaire, 1 (1984), pp.~109--145.

\bibitem{lmn:cpam04}
{\sc M.~Lucia, C.~B. Muratov, and M.~Novaga}, {\em Linear vs. nonlinear
  selection for the propagation speed of the solutions of scalar
  reaction-diffusion equations invading an unstable equilibrium}, Commun. Pure
  Appl. Math., 57 (2004), pp.~616--636.

\bibitem{lmn:arma08}
\leavevmode\vrule height 2pt depth -1.6pt width 23pt, {\em Existence of
  traveling waves of invasion for {Ginzburg-Landau-type} problems in infinite
  cylinders}, Arch. Rational Mech. Anal., 188 (2008), pp.~475--508.

\bibitem{lunardi}
{\sc A.~Lunardi}, {\em Analytic semigroups and optimal regularity in parabolic
  problems}, vol.~16 of Progress in Nonlinear Differential Equations and their
  Applications, Birkh\"auser, Basel, 1995.

\bibitem{mallordy95}
{\sc J.-F. Mallordy and J.-M. Roquejoffre}, {\em A parabolic equation of the
  {KPP} type in higher dimensions}, SIAM J. Math. Anal., 26 (1995), pp.~1--20.

\bibitem{m:dcdsb04}
{\sc C.~B. Muratov}, {\em A global variational structure and propagation of
  disturbances in reaction-diffusion systems of gradient type}, Discrete Cont.
  Dyn. S., Ser. B, 4 (2004), pp.~867--892.

\bibitem{mn:cms08}
{\sc C.~B. Muratov and M.~Novaga}, {\em Front propagation in infinite
  cylinders. {I. A} variational approach}, Comm. Math. Sci., 6 (2008),
  pp.~799--826.

\bibitem{mn:cvar08}
\leavevmode\vrule height 2pt depth -1.6pt width 23pt, {\em Front propagation in
  infinite cylinders. {II. The} sharp reaction zone limit}, Calc. Var. PDE, 31
  (2008), pp.~521--547.

\bibitem{protter}
{\sc M.~H. Protter and H.~F. Weinberger}, {\em Maximum principles in
  differential equations}, Springer-Verlag, New York, 1984.

\bibitem{risler08}
{\sc E.~Risler}, {\em Global convergence toward traveling fronts in nonlinear
  parabolic systems with a gradient structure}, Ann. Inst. H. Poincar\'e Anal.
  Non Lin\'eaire, 25 (2008), pp.~381--424.

\bibitem{roquejoffre92}
{\sc J.-M. Roquejoffre}, {\em Stability of travelling fronts in a model for
  flame propagation. {II}. {N}onlinear stability}, Arch. Rational Mech. Anal.,
  117 (1992), pp.~119--153.

\bibitem{roquejoffre94}
\leavevmode\vrule height 2pt depth -1.6pt width 23pt, {\em Convergence to
  travelling waves for solutions of a class of semilinear parabolic equations},
  J. Differ. Equations, 108 (1994), pp.~262--295.

\bibitem{roquejoffre97}
\leavevmode\vrule height 2pt depth -1.6pt width 23pt, {\em Eventual
  monotonicity and convergence to traveling fronts for the solutions of
  parabolic equations in cylinders}, Ann. Inst. H. Poincar\'e Anal. Non
  Lin\'eaire, 14 (1997), pp.~499--552.

\bibitem{rothe81}
{\sc F.~Rothe}, {\em Convergence to pushed fronts}, Rocky Mountain J. Math., 11
  (1981), pp.~617--633.

\bibitem{sattinger76}
{\sc D.~H. Sattinger}, {\em On the stability of waves of nonlinear parabolic
  systems}, Advances in Math., 22 (1976), pp.~312--355.

\bibitem{struwe}
{\sc M.~Struwe}, {\em Variational methods: applications to nonlinear partial
  differential equations and Hamiltonian systems}, Springer, Berlin, 2000.

\bibitem{taylor3}
{\sc M.~E. Taylor}, {\em Partial Differential Equations III: Nonlinear
  Equations}, Springer-Verlag, Berlin, 1996.

\bibitem{uchiyama77}
{\sc K.~Uchiyama}, {\em The behavior of solutions of some non-linear diffusion
  equations for large time}, J. Math. Kyoto Univ., 18 (1977), pp.~453--508.

\bibitem{vega93}
{\sc J.~M. Vega}, {\em The asymptotic behavior of the solutions of some
  semilinear elliptic equations in cylindrical domains}, J. Differ. Equations,
  102 (1993), pp.~119--152.

\bibitem{vega93jmma}
\leavevmode\vrule height 2pt depth -1.6pt width 23pt, {\em On the uniqueness of
  multidimensional travelling fronts of some semilinear equations}, J. Math.
  Anal. Appl., 177 (1993), pp.~481--490.

\bibitem{volpert}
{\sc A.~I. Volpert, V.~A. Volpert, and V.~A. Volpert}, {\em Traveling wave
  solutions of parabolic systems}, AMS, Providence, 1994.

\bibitem{xin00}
{\sc J.~Xin}, {\em Front propagation in heterogeneous media}, SIAM Review, 42
  (2000), pp.~161--230.

\end{thebibliography}

\end{document}